\documentclass[11pt,leqno]{amsart}
\usepackage{amsthm,amsfonts,amssymb,amsmath,oldgerm}
\numberwithin{equation}{section}
\usepackage{fullpage}
\usepackage{amsmath}
\usepackage{amsfonts}
\usepackage{amssymb}
\usepackage{setspace}
\usepackage{stmaryrd}
\usepackage{mathrsfs}
\usepackage{fancyhdr}
\usepackage{graphicx}
\usepackage{tikz}

\usepackage{psfrag}

\usepackage[top=30mm,bottom=30mm,left=25mm,right=25mm,a4paper]{geometry}

\renewcommand\d{\partial}
\newcommand\dD{\textrm{d}}
\newcommand\eD{\textrm{e}}
\newcommand\iD{\textrm{i}}

\newcommand{\Id}{{\rm Id}}
\newcommand{\I}{{\rm I}}

\newcommand{\intl}{[\![}
\newcommand{\intr}{]\!]}

\newcommand\br{\begin{remark}}
\newcommand\er{\end{remark}}
\newcommand\bp{\begin{pmatrix}}
\newcommand\ep{\end{pmatrix}}
\newcommand{\be}{\begin{equation}}
\newcommand{\ee}{\end{equation}}
\newcommand\ba{\begin{equation}\begin{aligned}}
\newcommand\ea{\end{aligned}\end{equation}}

\newcommand{\beg}{\begin{example}}
\newcommand{\eeg}{\end{exaplem}}
\newcommand{\bpr}{\begin{proposition}}
\newcommand{\epr}{\end{proposition}}
\newcommand{\bt}{\begin{theorem}}
\newcommand{\et}{\end{theorem}}
\newcommand{\bc}{\begin{corollary}}
\newcommand{\ec}{\end{corollary}}
\newcommand{\bl}{\begin{lemma}}
\newcommand{\el}{\end{lemma}}
\newcommand{\bd}{\begin{definition}}
\newcommand{\ed}{\end{definition}}
\newcommand{\brs}{\begin{remarks}}
\newcommand{\ers}{\end{remarks}}

\newtheorem{theorem}{Theorem}[section]
\newtheorem{proposition}[theorem]{Proposition}
\newtheorem{corollary}[theorem]{Corollary}
\newtheorem{lemma}[theorem]{Lemma}

\theoremstyle{remark}
\newtheorem{remark}[theorem]{Remark}
\theoremstyle{definition}
\newtheorem{definition}[theorem]{Definition}

\newtheorem{example}[theorem]{Example}

\newcommand\R{\mathbf R}
\newcommand\C{\mathbf C}
\newcommand{\N}{\mathbf N}
\newcommand{\Z}{\mathbf Z}

\newcommand{\tq}{\tilde q}

\newcommand{\uad}{\uu^{\textrm{ad}}}
\newcommand{\Vad}{V^{\textrm{ad}}}

\newcommand\bB{{\mathbf B}}

\newcommand\bD{{\mathbf D}}

\newcommand\bJ{{\mathbf J}}

\newcommand\bM{{\mathbf M}}

\newcommand\bfdelta{{\mathbf \delta}}

\newcommand\ubM{{\underline \bM}}

\newcommand\uE{{\underline E}}

\newcommand\uM{{\underline M}}

\newcommand\uU{{\underline U}}

\newcommand\uc{{\underline c}}

\newcommand\uk{{\underline k}}

\newcommand\uup{{\underline p}}

\newcommand\uu{{\underline u}}

\newcommand\uom{{\underline \omega}}

\newcommand\cB{{\mathcal B}}
\newcommand\cC{{\mathcal C}}

\newcommand\cF{{\mathcal F}}

\newcommand\cH{{\mathcal H}}

\newcommand\cL{{\mathcal L}}

\newcommand\cO{{\mathcal O}}

\newcommand\tbD{\widetilde{\bD}}

\title{
Spectral validation of the Whitham equations for periodic waves of lattice dynamical systems
}

\author{Bu\u{g}ra Kabil}
\address{Bu\u{g}ra Kabil, 
Institute of Applied Analysis and Numerical Simulation, Pfaffenwaldring 57, 70569 Stuttgart, University of Stuttgart, Germany}
\email{{\tt Bugra.Kabil@mathematik.uni-stuttgart.de}}
\author{L.Miguel Rodrigues}
\address{L.Miguel Rodrigues,
Universit\'e de Lyon,
CNRS UMR 5208,
Universit\'e Lyon 1,
INRIA \'EP Kaliffe,
Institut Camille Jordan,
43 bd 11 novembre 1918;
F-69622 Villeurbanne cedex}
\email{{\tt Rodrigues@math.univ-lyon1.fr}}
\thanks{Research of L.M.R. was partially supported by the ANR project BoND ANR-13-BS01-0009-01.}

\begin{document}

\begin{abstract}
In the present contribution we investigate some features of dynamical lattice systems near periodic traveling waves. First, following the formal averaging method of Whitham, we derive modulation systems expected to drive at main order the time evolution of slowly modulated wavetrains. Then, for waves whose period is commensurable to the lattice, we prove that the formally-derived first-order averaged system must be at least weakly hyperbolic if the background waves are to be spectrally stable, and, when weak hyperbolicity is met, the characteristic velocities of the modulation system provide group velocities of the original system. Historically, for dynamical evolutions obeying partial differential equations, this has been proved, according to increasing level of algebraic complexity, first for systems of reaction-diffusion type, then for generic systems of balance laws, at last for Hamiltonian systems. Here, for their semi-discrete counterparts, we give at once simultaneous proofs for all these cases. Our main analytical tool is the discrete Bloch transform, a discrete analogous to the continuous Bloch transform. Nevertheless, we needed to overcome the absence of genuine space-translation invariance, a key ingredient of continuous analyses.
\end{abstract}

\date{\today}
\maketitle




\tableofcontents


\section{Introduction}\label{s:introduction}

A common strategy to tackle the analysis of large time behavior of dynamical systems focuses on coherent structures, often playing the role of organizing centers for the time evolution. We follow here this line of investigation for lattice dynamical systems and restrict our attention to periodic traveling waves.

By definition, evolution of lattice dynamical systems is continuous in time but discrete in space. We consider here particular instances of those where the lattice is one-dimensional and thus may be assumed to be $\Z$, and the evolution obeys a differential equation 
$$
\dfrac{\dD}{\dD t}U\ (t)\ =\ \cF(U(t))
$$
where the unknown $U$ associates to any time $t\in\R$ an $\R^d$-valued sequence $U(t)\in (\R^d)^\Z$ ($d\in\N^*$ being a fixed dimension) and $\cF$ is a vector-field on $(\R^d)^\Z$ that preserves the linear space of finitely-supported sequences and acts smoothly on it. Alternatively one may view the evolution as given by an infinite number of differential equations on $\R^d$-valued functions coupled by a finite-range interaction. We expect that our results could actually deal with nonlocalized interaction satisfying certain short-range decaying assumptions but we choose to ignore those here to keep technicalities as low as possible. 

There are at least two ways in which such systems arise. They may come directly from modeling at a discrete level, as for neural networks, granular media, crystals, biological molecules, optical waveguides, chains of coupled oscillators... Alternatively they may emerge from the discretization in space --- often called semi-discretization --- of systems of partial differential equations, that generate fully continuous dynamical systems, 
In the latter case, mark that although a study of dynamical evolutions generated by fully discrete schemes including boundary conditions would be of much more direct practical use, such a study appears as a daunting task and it is a common belief that still relevant pieces of information are obtained from investigation of extended semi-discretized schemes. The reader looking for panoramas on lattice dynamical systems is referred to \cite{Chow-MalletParet,Chow-MalletParet-vanVleck,MalletParet_survey,Pankov,Kevrekidis_review}. Specially designed technical tools of wide application range may be found in \cite{Rustichini_linear,Rustichini_Hopf,Chow-MalletParet-Shen,MalletParet_Fredholm,MalletParet_global-structure,Iooss,Kapitula-Kevrekidis,Cramer-Latushkin,James_center-manifold,Friesecke-Pego-I,Friesecke-Pego-II,Friesecke-Pego-III,Friesecke-Pego-IV,Hupkes-VerduynLunel_center-manifold,Hupkes-VerduynLunel_center-manifold_periodic,Hupkes-VerduynLunel_bifurcations}. Especially the latter references include key-ingredients to prove existence of traveling waves for various class of systems, a fact that is taken as an assumption in the present piece of work.

We are interested here in a particular area of the dynamics, namely what occurs in neighborhoods of periodic traveling waves. Compared to other coherent structures of lattice dynamical systems --- such as fronts, kinks, pulses, shocks, solitary waves or breathers ---, periodic waves have received relatively less attention. In particular, while the stability analysis of some of the former patterns seems to have reached some maturity --- see for instance \cite{MalletParet_global-structure,Friesecke-Pego-I,Friesecke-Pego-II,Friesecke-Pego-III,Friesecke-Pego-IV,BenzoniGavage-Huot-Rousset,Beck-Hupkes-Sandstede-Zumbrun,Hupkes-Sandstede-III} ---, the authors of the present paper are not aware of even a single example of a comprehensive stability study for periodic waves\footnote{In a sense analogous to \cite{Schneider-SH,Schneider-proc,Mielke-Schneider-Uecker,JNRZ-conservation,R} for dissipative \emph{continuous} systems.} of some lattice dynamical system. Nevertheless, the interested reader may benefit from looking at \cite{Keener,Tarollo-Terracini,Bressloff-Rowlands,Wu-Zou,Filip-Venakides,Pankov-Pfluger,Iooss-Kirchgassner,Feckan-Rothos_homoclinics,Pankov,Carpio,Feckan-Rothos_kink,Bak,Guo-Lamb-Rink,Makita,Venney-Zimmer,Herrmann_discrete-scalar-conservation-laws,James,Betti-Pelinovsky,Lin-Huang-Cheng}, that are at least partly dedicated to periodic waves of lattice dynamical systems, and mostly focusing on proving their existence.

To set our precise framework, let us specify that we say that, for systems as above, a traveling wave, sometimes called a uniformly sliding solution, --- that is, a solution $U$ with special form $U(t)=(u(j-ct))_{j\in\Z}$ for some speed $c\in\R$ and some profile $u:\R\to\R^d$ --- is said to be periodic if the corresponding profile $u$ is itself periodic. In this case we rescale phase variable of the profile to ensure that the profile period is one. Explicitly any periodic traveling wave $U$ is written as $U(t)=(u(k\,j+\omega\,t))_{j\in\Z}$ with a one-periodic profile $u$. This brings out spatial wavenumber $k\in\R^*$ and time frequency $\omega=-k\,c$. Mark that up to now we are not assuming that the spatial wavenumber $k$ is rational. When its is not the case the spatial periodicity of the profile is not a priori easily observed on the solution itself, though when standing waves are discarded, that is when restricting to $c\neq0$ this periodicity is apparent in time, the solution being $1/\omega$-periodic in time. We refer the reader to \cite{Herrmann_discrete-scalar-conservation-laws} for further elaboration on this comment including fascinating illustrating pictures. This already stresses that, in contrast with analogous issues for continuous dynamical systems, analysis of the behavior near standing waves follows from different kind of arguments. Indeed, while profile equations for genuinely-traveling waves are \emph{differential} functional equations, profile equations for standing waves are \emph{algebraic} functional equations. This results in the fact that even when traveling and standing waves coexist the standing-wave limit $c\to0$ is very singular. For this reason, our analysis focused on on genuinely-traveling waves do not extend in a straightforward way neither to breathers nor to steady spatially periodic solutions.

Our goals are, first, to derive --- on a formal basis --- a system of partial differential equations that is expected to describe the time evolution of slowly modulated wavetrains of lattice dynamical systems, second, to elucidate --- on mathematical grounds ---  what are, at the spectral level, the connections between the linearization about a constant state of the formally-derived modulation systems and the linearization about a periodic wave of our initial lattice system of equations. Our formal derivation is in many ways very classical and it is well-known as one of the possibilities to recover a continuous description from a discrete model. The reader is referred to \cite{Hays-Levermore-Miller,Hays-Levermore-Miller_erratum,Dreyer-Herrmann-Mielke,Dreyer-Herrmann-Rademacher,Giannoulis-Hermann-Mielke,Dreyer-Herrmann} for various perspectives on this point of view. Our main motivation here is to provide a rigorous spectral validation, that should be compared with some contributions of \cite{DSSS,Serre,Noble-Rodrigues,Benzoni-Noble-Rodrigues} on partial differential equations, essentially as \cite{Friesecke-Pego-I,Friesecke-Pego-II,Friesecke-Pego-III,Friesecke-Pego-IV} may be thought as a discrete counterpart of \cite{Pego-Weinstein-eigenvalues,Pego-Weinstein-asymptotic-stability}.

As a direct consequence of our results stems the proof that characteristic speeds of averaged systems do provide group velocities for the original one. Mark that even if we illustrate our strategy mostly on systems obtained from semi-discretization of continuous models we provide a proof at the discretized level so that, in these particular cases, group velocities implicitly depend on the mesh size and of the particular choice of discretization. Given the fundamental role of a precise knowledge of group velocities in a quantitative description of the dynamics, and since accurate reproduction of dispersion relations obtained by linearizing about constant states has served for a long time now to discriminate between numerical schemes with, from other respects, similar performances, we do not exclude that in a near future comparison between discrete group velocities and continuous ones could serve similar purposes.

Our contribution may seem relatively modest in that the proved connection is purely at the spectral level. Yet on the other hand our strategy is very robust as we emphasize by treating here in a single place various classes systems that were considered at the continuous level in separate contributions. Moreover, as our proofs provide relations for eigenmodes --- including eigenvectors and not restricting to eigenvalues ---, we expect that it could be directly useful for a dynamical validation of averaged systems in the large-time limit, as, in continuous cases, \cite[Section~4.2]{DSSS} for \cite{SSSU,JNRZ-RD2}, \cite{Noble-Rodrigues} for \cite{JNRZ-conservation} or \cite{Benzoni-Noble-Rodrigues} for \cite{R_linKdV}.

From our point of view, our main restriction is that in the spectral validation we only consider periodic waves whose period is a multiple of the generator of the lattice. This stringent restriction does not fit well in the spirit of slowly modulated wavetrains that involve a continuum of such periodic waves but it enables us to use techniques originating in the study of differential operators with periodic coefficients --- the classical Floquet theory --- rather than those for quasi-periodic coefficients. For a glimpse at the fascinating but technically daunting quasi-periodic theory the reader is referred to \cite{Eliasson_I,Eliasson_II,Eliasson_review,Eliasson_discrete-review} where is analyzed the spectrum of operators similar to those obtained by linearizing discrete equations about a general spatially-periodic stationary solution.

From a technical point of view, in our analysis of lattice systems, departures from the continuous strategy originate in the loss of a genuine space-translation invariance of the original system --- see Remark~\ref{space-translation-invariance} though. To balance this, we make an extensive use of the remaining time-translation invariance and of the space-translation invariance of the profile equation. In particular, instead of moving to a mobile frame in which the wave is at rest and analyzing the spectrum of the generator of the linearized dynamics, we study directly Floquet multipliers of the linearized time evolution, that is, we study the spectrum of the map that encodes the evolution over a time $1/\omega$ according to the linearized dynamics. Likewise, when dealing with Hamiltonian systems, we replace in our arguments the missing momentum conservation law with local conservation of energy. On the other hand, we may still use a normal, Floquet exponent by Floquet exponent, Bloch-wave decomposition by relying upon the discrete Bloch transform. Then to examine the spectrum of resulting discrete Bloch symbols we perform here a direct spectral perturbation analysis "\`a la Kato" \cite{Kato} as in \cite{Noble-Rodrigues,Benzoni-Noble-Rodrigues} for continuous systems.  Alternatively we could also have introduced a suitable Evans function, as done in \cite{Serre} for continuous dynamics and in \cite{Kapitula-Kevrekidis,Benzoni-Gavage,Cramer-Latushkin} for other classes of coherent structures in lattice systems. For a detailed discussion of respective advantages of proofs by direct spectral perturbation or by Evans function expansions see \cite{R}.

The rest of the present paper is organized as follows. In the second section, examining formal expansions, we derive modulated equations. In the third section we introduce tools from functional analysis --- including the discrete Bloch transform --- that we shall subsequently use to analyze the spectrum of the linear one-period evolution operator. In Section~\ref{s:spectrum} we state and prove our main results concerning the spectral validation of the previously obtained averaged systems. We close this paper with further comments and open questions. In Sections \ref{s:Whitham} \& \ref{s:spectrum}, we analyze three classes of systems in increasing order of algebraic complexity. First, we consider a lattice system of reaction-diffusion type, in the sense that the system does not support any local conservation law. In this case, the averaged evolution obeys a scalar equation and both derivation and spectral validation are considerably simpler so that we also provide high-order versions of those. It is actually those higher-order versions that are needed in the dynamical large-time validation of the slow modulation scenario. After that, we examine a general system that includes built-in conservation laws. The averaged dynamics is now prescribed by a genuine system and the spectral validation includes spectral perturbation of a Jordan block, hence requires a preliminary desingularization. At last, we consider a lattice system of Hamiltonian type that comes with an "extra" conservation law encoding conservation of energy. Again, all the corresponding theorems, Theorem~\ref{propRD}, Theorem~\ref{propMIX} and Theorem~\ref{propHAM}, are stated and proved in Section~\ref{s:spectrum}.

\vspace{2em}

\noindent \textbf{Notation.} Elements $U\in (\C^d)^\Z$ of $(\C^d)^\Z$ are implicitly assumed to have coordinates $U_j=U(j)\in \C^d$, $j\in\Z$. Likewise for $U\in (\C^d)^{\Z\times\R}$, $U_j(t)=U(j,t)$, for $(j,t)\in\Z\times\R$, and $U(t)=U(\cdot,t)\in (\C^d)^\Z$, for $t\in\R$. All along $\mathbf{T}$ denotes the left shift operator on $(\C^d)^\Z$, 
$$\left(\mathbf{T}\, U\right)(j)\ =\ U(j+1)\,,\qquad\textrm{for }U\in (\C^d)^\Z\,.$$
The set $\N$ is the set of nonnegative integers, and, for $(M,N)\in\Z$, $\intl M,N\intr=[M,N]\cap \Z$.  
The resolvent set and the spectrum set of an operator are denoted respectively by $\rho(\cdot)$ and $\sigma(\cdot)$. 



\section{Formal derivation of averaged equations}\label{s:Whitham}

In this section we derive, from formal considerations, averaged modulation systems for three classes of nonlinear
lattice systems, starting with the simplest case of lattice systems of reaction-diffusion type. To do so, we will follow the two-timing method introduced by Whitham \cite{Whitham}. Only afterwards, in Section~4 , shall we prove that the original dynamical evolution is indeed related to these modulation equations.

\subsection{Reaction-diffusion case}\label{s:Whitham-RD}

We consider a discretization of the scalar reaction-diffusion equation for $u:\R\times\R^+\rightarrow \R^d$
\begin{eqnarray}\label{RDeq}
\d_tu\ =\ \Delta u + f(u)
\end{eqnarray}
where $f:\R^d\rightarrow \R^d$ is a given (in general nonlinear) smooth function. Using a centered difference scheme to discretize the spatial Laplacian operator, we obtain the following system of coupled ordinary differential equations
\begin{eqnarray}\label{disRD}
\dfrac{\dD}{\dD t}U_j(t)= \mu[ U_{j+1}(t)-2U_j(t)+U_{j-1}(t) ] + f(U_j(t)), \qquad j\in\Z,\, t>0,
\end{eqnarray}
where $ \mu\in\R$ is a given fixed constant. 
To consider it in the abstract form of dynamical lattice systems, we introduce $F$ the nonlinear operator on $U\in(\C^d)^{\Z}$  associated with $f$ and given by $F(U)=(f(U_j))_{j\in\Z}$. Then, we write system \eqref{disRD} for $U(\cdot)=(U_j(\cdot))_{j\in\Z}$ as
\begin{eqnarray}\label{nonlLA}
\dfrac{\dD}{\dD t}U(t)= \mu\,(\mathbf{T}-2\Id+\mathbf{T}^{-1})\,U(t)+F(U(t))), \qquad t>0\,.
\end{eqnarray}

A profile $u$ generates through $U(t)=(u(kj+\omega\,t))_{j\in\Z}$, a periodic traveling wave to \eqref{nonlLA} with spatial wavelength $k\neq0$ and time frequency $\omega$ if and only if
\begin{eqnarray}\label{funcODE}
\omega u'(\zeta)= \mu[\,u(\zeta+k)-2\,u(\zeta)
+u(\zeta-k)\,]+f(u(\zeta)),\qquad \zeta\in\R\,,
\end{eqnarray}
and $u$ is one-periodic
\be\label{period-one}
u(\zeta+1)=u(\zeta),\qquad \zeta\in\R\,.
\ee

To obtain (even formal) pieces of information about a given periodic wave, we need a precise knowledge of neighboring periodic waves. With this respect we make the following "non degeneracy" assumption:
\be
\label{A_RD}
\begin{array}{r}
\textrm{Periodic traveling waves of \eqref{nonlLA} --- identified when coinciding up to time}\\ 
\textrm{translation --- form a smooth manifold of dimension one, regularly parametrized by $k$.}
\end{array}\tag{A}
\ee
For latter use, we denote by $\uu^k(\cdot),\omega(k)$ such a parametrization by $k$. We refer the reader to references given in the introduction for proofs of existence results. We also stress that our assumption is consistent with both the continuous limit leading to \eqref{RDeq} and similar considerations in \cite{Hays-Levermore-Miller,Hays-Levermore-Miller_erratum,Dreyer-Herrmann-Mielke,Dreyer-Herrmann-Rademacher,Giannoulis-Hermann-Mielke,Dreyer-Herrmann}.

Our formal \emph{ansatz} is of two-scale type, one scale slow --- but otherwise arbitrary ---, involving macroscopic variables, another one fast but oscillatory, hence described through an oscillation phase. Explicitly we want to gain some insight about the asymptotic behavior --- when the slow frequency $\epsilon$ goes to zero --- of families of solution $(U^{(\epsilon)}(\cdot))$ expanding as 
\begin{eqnarray}
U^{(\epsilon)}_j(t)\ =\ u^{(\epsilon)}\left(\underbrace{\varepsilon j}_{X},\underbrace{\varepsilon t}_{T},\underbrace{\dfrac{\phi^{(\epsilon)}(\epsilon j,\epsilon t)}{\epsilon}}_{\Theta}\right)
+\cO(\epsilon^2),
\end{eqnarray} 
with phases $\phi^{(\epsilon)}$
\begin{eqnarray}
\phi^{(\epsilon)}(X,T)
\ =\ \phi_0\left(X,T\right)
+\epsilon\, \phi_1\left(X,T\right)+\epsilon^2\, \phi_2\left(X,T\right)
+\cO(\epsilon^3)\,,
\end{eqnarray}
and profiles $u^{(\epsilon)}$ one-periodic in the third variable $\Theta$
$$
u^{(\epsilon)}(X,T,\Theta)
\ =\ u_0\left(X,T,\Theta\right)
+\epsilon\, u_1\left(X,T,\Theta\right)
+\cO(\epsilon^2)\,.
$$
This yields the following relation
\begin{eqnarray*}
&\epsilon&\d_T u^{(\epsilon)}\left(X,T,\dfrac{\phi^{(\epsilon)}(X,T)}{\epsilon}\right)+(\d_T \phi^{(\epsilon)}(X,T) ) \cdot \d_\Theta u^{(\epsilon)}\left(X,T,\dfrac{\phi^{(\epsilon)}(X,T)}{\epsilon}\right)\\
&=& \mu\,\left[\,u^{(\epsilon)} \left(X+\epsilon,T,\dfrac{\phi^{(\epsilon)}(X+\epsilon,T)}{\epsilon}\right)-
2u^{(\epsilon)}\left(X,T,\dfrac{\phi^{(\epsilon)}(X,T)}{\epsilon}\right)
+u^{(\epsilon)}\left(X-\epsilon,T,\dfrac{\phi^{(\epsilon)}(X-\epsilon ,T)}{\epsilon}\right)\,\right]\\
&+&f\left(u^{(\epsilon)}\left(X,T,\dfrac{\phi^{(\epsilon)}(X,T)}{\epsilon}\right)\right)+\cO(\epsilon^2)\,.
\end{eqnarray*}
Usual identification of powers of $\epsilon$ and replacement of the phase $\phi_0(X,T)/\epsilon+\phi_1(X,T)$ with $\Theta$ then provide at order $\epsilon^0$
\begin{eqnarray*}
&(\d_T \phi_0(X,T) )& \d_\Theta u_0 \left( X,T,\Theta \right)\\
&=&\mu\left[\,u_0 \left( X,T,\Theta+\d_X \phi_0(X,T)\right)-2u_0 \left( X,T,\Theta \right)
+u_0 \left( X,T,\Theta-\d_X \phi_0(X,T) \right)\,\right]\\
&+&f(u_0 \left( X,T,\Theta \right))\,,
\end{eqnarray*}
which is, in the $\Theta$ variable, the functional profile equation with wavenumber and time frequency
\begin{eqnarray}
k_0 (X,T) := \d_X \phi_0(X,T) \qquad \text{and} \qquad \omega_0(X,T):= \d_T \phi_0(X,T)\,.   
\end{eqnarray}
Hence, under a suitable normalization of parametrization, without loss of generality
\be\label{slowmodRD}
u_0(X,T,\Theta)\ =\ \uu^{k_0(X,T)}\,(\Theta)
\ee
where evolution of the local wavenumber obeys 
\be\label{W_RD}
\d_T\,k_0\ -\ \d_X(\omega(k_0))\ =\ 0\,,
\ee
stemming from Schwarz' identity $\d_T\d_X\phi_0=\d_X\d_T\phi_0$.

As a result of our heuristics, we obtain the conjecture that slow/oscillatory solutions evolve at main order according to a slow modulation scenario
$$
U^{(\epsilon)}_j(t)\ =\ \uu^{k_0(\epsilon\,j,\epsilon\,t)}\left(\dfrac{\phi_0(\epsilon j,\epsilon t)}{\epsilon}+\phi_1(\epsilon j,\epsilon t)\right)
\,+\,\cO(\epsilon),
$$
that is, locally, at scale $1$, the solutions look like one member of the periodic traveling wave family, but with local parameters --- phase shift $\phi_1$ and wavenumber $k_0$ --- evolving on a slow scale $1/\epsilon$. In this description local phase and local spatial wavenumber are tied by $k_0=\d_X\phi_0$ and the evolution of the local wavenumber $k_0$ obeys at main order \eqref{W_RD}. In particular, unsurprisingly, in this simple context, linear group velocity around a given wave $\uU$ generated by a profile $\uu^\uk$ is expected to be $\omega'(\uk)$.

In this simple case where the family of periodic waves is one-dimensional and hence modulation equation \eqref{W_RD} is scalar, algebra required to raise the order of description remains relatively simple so that we explain it now. It requires a more precise account of assumption \eqref{A_RD}. In assuming a regular parametrization by $k$, we include that for any wave under consideration --- $\uu$ with wavenumber $\uk$ and frequency $\uom$ --- the operator $\cL$ acting on $L_{per}^2([0,1])$ with domain $H^1_{per}([0,1])$ through
\be\label{op_per}
\cL u\ =\ -\uom\, u'+\mu[\,u(\cdot+\uk)-2u+u(\cdot-\uk)\,]+\dD f(\uu(\cdot))\,u
\ee
has one-dimensional kernel. By time translation invariance, this kernel is thus $\C\,\uu'$. Moreover since, as a (relatively) compact perturbation of $-\omega (\ )'$, $\cL$ is Fredholm of index $0$ --- when considered as usual as a bounded operator on its domain endowed with its graph norm ---, we conclude the existence of a $\uad_{\uk}\in H^1_{per}([0,1])$ spanning the kernel of the adjoint $\cL^*$ of $\cL$ and such that $\langle\uad_{\uk},\uu'\rangle=1$ (in the $L_{per}^2([0,1])$ sense). In particular by differentiating the profile \eqref{funcODE} and taking a scalar product with $\uad_\uk$, one receives
\be\label{lingroup_RD}
\omega'(\uk)\ =\ \langle\uad_\uk,\mu[\uu'(\cdot+\uk)-\uu'(\cdot-\uk)]\rangle\,.
\ee
Moreover, $\uad_{\uk}$ allows us to give a precise account of what is the appropriate normalization of parametrization that we have alluded to above. Indeed we explicitly require
\be
\label{normal_RD}\langle\uad_{\uk},\d_k\uu^\uk\rangle\ =\ 0\tag{N}
\ee
in order to enforce the distinct respective roles of wavenumber and phase shift. The above condition may be achieved by shifting the original $\uu^k$ in a suitable $k$-dependent way.

After these preliminaries we may go back to the identification of terms of order $\epsilon^1$. By dropping all $(X,T)$ dependencies, we receive
\begin{eqnarray*}
-\cL u_1&+&\d_T\phi_1\d_\Theta u_0\,-\,\d_X\phi_1\,\mu\,[\d_\Theta u_0(\cdot+k_0)-\d_\Theta u_0(\cdot-k_0)]\\
&=&-\d_Tk_0 \d_k \uu^{k_0}+\mu\d_Xk_0\,[\d_k \uu^{k_0}(\cdot+k_0)-\d_k \uu^{k_0}(\cdot-k_0)]\\
&&+\frac12\mu\d_Xk_0\,[\d_\Theta u_0(\cdot+k_0)+\d_\Theta u_0(\cdot-k_0)]
\end{eqnarray*}
where $\cL$ is here an $(X,T)$-parametrized version of original $\cL$, associated with $u_0$ and acting in the $\Theta$-variable. Averaging against $\uad_{k_0}$ yields
$$
\d_T\phi_1\,-\,\omega'(k_0)\,k_1\ =\ d(k_0)\,\d_Xk_0\,.
$$
where, for concision's sake, we have introduced a diffusion coefficient
\be\label{linvisc}
d(k)\ =\ \langle\uad_{k},\mu\,[\d_k \uu^{k}(\cdot+k)-\d_k \uu^{k}(\cdot-k)]+
\tfrac12\mu\,[\d_\zeta\uu^k(\cdot+k)+\d_\zeta \uu^k(\cdot-k)]\rangle\,.
\ee

As a refined conclusion, we obtain that slow/oscillatory solutions evolve at main order according to
$$
U^{(\epsilon)}_j(t)\ =\ \uu^{k(\epsilon\,j,\epsilon\,t)}\left(\dfrac{\phi(\epsilon j,\epsilon t)}{\epsilon}\right)
\,+\,\cO(\epsilon)\,,
$$
with local phase and local spatial wavenumber tied by $k=\d_X\phi$ and the evolution of the local wavenumber $k$ obeys at second order 
\be\label{W2_RD}
\d_T k\,-\,\d_X(\omega(k))\ =\ \epsilon\,\d_X(d(k)\d_Xk)\,.
\ee
Mark that in this second-order description, there is no more undetermined local phase shift and that, when present, large-time diffusive decay or anti-diffusive growth may be captured by equation~\eqref{W2_RD}.

\subsection{General mixed case}\label{s:Whitham-conservation}

The main point about discrete reaction-diffusion systems considered in the previous subsection is that they do not support any built-in conservation law. Assumption~\ref{A_RD} then requires that indeed no hidden conservation is present. This leads to a scalar modulation behavior encoded by \eqref{W_RD} or \eqref{W2_RD}. We now relax this stringent assumption to consider a system mixing conservative equations with non conservative ones.

At the continuous level such structure is ubiquitous. For instance it naturally emerges from the modeling of the evolution of isentropic compressible flows undergoing external forces by the Navier-Stokes or Euler systems. Typically, in the former case one receives for the evolution of mass density $\rho$ and velocity field $u$ a system in the form
\be\label{CONeq}
\begin{array}{rcl}
\d_t\rho+\d_x(\rho\,u)&=&0\\
\d_t(\rho u)+\d_x(\rho u^2+P(\rho))&=&\d_x(\nu(\rho)\d_x u)+g(\rho,u),
\end{array}
\ee
where $P$, $\nu$ and $g$ provide respectively the pressure law, Lam\'e viscosity coefficient and external forces.

Here we consider for $U=(R,W)^T$ --- valued in $(\R^d)^Z=(\R^{d_1})^Z\times(\R^{d_2})^Z$ --- systems of the following general form
\be\label{disCON}
\begin{array}{rcl}
\dfrac{\dD}{\dD t}R(t)\,+\,\bD_1 (f_r(R(t),W(t)))&=&0 \\[1em]
\dfrac{\dD}{\dD t}W(t)\,+\,\bD_2 (f_w(R(t),W(t)))&=&\bD_3(B(R(t),W(t)) \bD_4 W(t))\,+\,g(R(t),W(t))\,,
\end{array}
\ee
where $f=(f_r,f_w)^T:\R^d\to\R^d$, $g:\R^d\to\R^{d_2}$ and $B:\R^d\to\cL(\R^{d_2})$ are identified with their component-wise counterparts actions on sequences, of respective type $(\R^d)^\Z\to(\R^d)^\Z$, $(\R^d)^\Z\to(\R^{d_2})^\Z$ and $(\R^d)^\Z\to \cL((\R^{d_2})^\Z)$, and $\bD_j$, $j=1,2,3,4$, are constant-coefficient discrete differential operators, $\bD_1$ being conservative in the sense that the kernel of $\bD_1^*$ (the adjoint of $\bD_1$ on $\ell^2(\Z;\R^{d_1})$) contains constant sequences. From now on, for definiteness' sake, we will restrict our attention to the case where for some $\eta>0$
\be\label{numOP}
\bD_1\ =\ \bD_4 \ =\ \eta\,(\mathbf{T}-\Id)\,,\quad \bD_2\ =\ \bD_3\ =\ \eta\,(\Id-\mathbf{T}^{-1})\,.
\ee 

As in the reaction-diffusion case, we make a "non degeneracy" assumption on the set of periodic traveling waves:
\be
\label{A_mixed}
\begin{array}{r}
\textrm{Periodic traveling waves of \eqref{disCON} --- identified when coinciding up to time}\\ 
\textrm{translation --- form a smooth manifold of dimension $d_1+1$, regularly parametrized}\\
\textrm{ by $(k,\bM)$, their wavenumber and the average-values of their $d_1$th first components.}
\end{array}\tag{A}
\ee
Explicitly $\bM=(M_1,\cdots,M_{d_1})$ with $M_j=\int_0^1 \eD_j\cdot \uu$ with $\eD_j$ the $j$th element of the canonical basis of $\R^d$. For latter use, we denote by $\uu^{k,\bM}(\cdot),\omega(k,\bM)$ a parametrization provided by the previous assumption. Again we refer the reader to references in the introduction for proofs of existence results.


Again, a formal \emph{ansatz}
\begin{eqnarray}
U^{(\epsilon)}_j(t)\ =\ u^{(\epsilon)}\left(\underbrace{\varepsilon j}_{X},\underbrace{\varepsilon t}_{T},\underbrace{\dfrac{\phi^{(\epsilon)}(\epsilon j,\epsilon t)}{\epsilon}}_{\Theta}\right)
+\cO(\epsilon^2),
\end{eqnarray}
with phases $\phi^{(\epsilon)}$
\begin{eqnarray}
\phi^{(\epsilon)}(X,T)
\ =\ \phi_0\left(X,T\right)
+\epsilon\, \phi_1\left(X,T\right)+\epsilon^2\, \phi_2\left(X,T\right)
+\cO(\epsilon^3)\,,
\end{eqnarray}
and profiles $u^{(\epsilon)}$ one-periodic in the third variable $\Theta$
$$
u^{(\epsilon)}(X,T,\Theta)
\ =\ u_0\left(X,T,\Theta\right)
+\epsilon\, u_1\left(X,T,\Theta\right)
+\cO(\epsilon^2)\,
$$
leads, under a suitable normalization of parametrization, to
$$
u_0(X,T,\Theta)\ =\ \uu^{(k_0,\bM_0)(X,T)}\,(\Theta)
$$
with
$$
k_0 (X,T) := \d_X \phi_0(X,T)\,,\qquad\omega_0(X,T):= \d_T \phi_0(X,T)
$$
and
$$
\bM_0(X,T):=\int_0^1\begin{pmatrix}\I_{d_1}&0_{d_1\times d_2}\end{pmatrix}u_0(X,T,\Theta)\,\dD\Theta\,.   
$$
But now the law of conservation of waves
$$
\d_T\,k_0\ -\ \d_X(\omega(k_0,\bM_0))\ =\ 0
$$
fails to describe completely the time-evolution.

To proceed, one collects terms of power $\epsilon^1$, replacing $\phi_0(X,T)/\epsilon+\phi_1(X,T)$ with $\Theta$, and, with usual implicit notation for coordinates splitting $u=(r,w)^T$, receives 
$$
-\cL u_1\ +\ \d_T(u_0)\ + \eta \mathbf{T}_{k_0}(\dD f(u_0)\cdot\d_Xu_0)
\ =\ \d_\Theta(\cdots)+\begin{pmatrix}0\\\cdots\end{pmatrix}\,,
$$
where $\cL$ is the operator associated with the linearized profile equation and $\mathbf{T}_{k_0}$ acts on functions of $\Theta$ through $\mathbf{T}_{k_0}(v)=v(\cdot+k_0)$.
Motivated by the fact that constant functions with values $e_j$, $j=1,\dots,d_1$, lie in the kernel of $\cL^*$, we average over $\Theta\in(0,1)$ the scalar product of the previous system with those $e_j$ and obtain
$$
\d_T\bM_0+\d_XF(k_0,\bM_0)\ =\ 0
$$
where
$$F(k,\bM)\ =\ \eta\int^1_0 f_r (\uu^{k,\bM}(\Theta))\,\dD\Theta \, .$$

Altogether we have derived the following averaged system 
\begin{eqnarray}
\label{WhithamCON}
\begin{cases}
\, \, \, \,  \d_T k-\d_X\omega(k,\bM)=0\\[0.2cm]
\d_T\bM+\d_XF(k,\bM)=0
\end{cases}
\end{eqnarray}
for the evolution of local parameters involved in a slow modulation description of the dynamics.

%

\subsection{Hamiltonian case}\label{s:Whitham-Hamiltonian}

Up to now we have implicitly assumed that our systems do not contain any "hidden" conservation law. In doing so we were motivated by our will to restrict to non-degenerate cases. However for special algebraic classes of systems 'hidden' conservation laws are indeed generic and we need to accommodate them. As a typical example we analyze systems that are (at least formally) of Hamiltonian type.
Incidentally, although we do not follow this path here, we stress that the presence of an extra structure also offers alternative ways of deriving the same modulation equations; see \cite{Whitham}.

To emphasize analogy with results on continuous dynamical systems in \cite{Benzoni-Noble-Rodrigues} we focus on a class of Hamiltonian lattice dynamical systems that include discrete counterparts to the Korteweg--de Vries equation and the Euler--Korteweg system. Explicitly we fix some Hamiltonian $\cH:(\R^d)^2\to\R$ and some skew-symmetric constant-coefficient discrete differential operator $\bJ$ in a conservative form, say for definiteness' sake $\bJ=\bD\,\bB$ with $\bB$ a symmetric matrix of size $d$ and $\bD\,=\,\frac{\eta}{2} \left(\mathbf{T}-\mathbf{T}^{-1}\right)$. Moreover we choose some other constant-coefficient discrete differential operator $\tbD$ (not necessarily skew-symmetric), say $\tbD=\eta\,(\mathbf{T}-\Id)$, for some $\eta>0$. 

Then we consider the following lattice dynamical system
\be\label{disHam}
\dfrac{\dD}{\dD t}U(t)\ =\ \bJ\,\bfdelta\cH[U(t)],
\ee
where $\bfdelta$ denotes a discrete Euler operator providing variational derivatives for functionals stemming from local functions of $(U,\tbD U)$,
\be\label{disEul}
\bfdelta\cH[U]\ =\ \nabla_U\cH (U,\tbD\,U)+\tbD^*\,\nabla_{\tbD\,U}\cH(U,\tbD U)\,.
\ee
By definition the previous system comes with a local conservation law for the discrete Hamiltonian
\begin{equation}\label{disConLoc}
\dfrac{\dD}{\dD t}[\cH (U,\tbD\,U)]=
\tbD\bigg[ 
\tfrac12\, \mathbf{T}^{-1}(\bfdelta\, \cH[U])\,\bB\, \bfdelta\, \cH[U]+\mathbf{T}^{-1}(\nabla_{\tbD\,U}\cH(U,\tbD\,U))\cdot \bJ\, \bfdelta\, \cH [U]
 \bigg]\,.
\end{equation}
Notice that to derive \eqref{disConLoc} from \eqref{disHam} we have used that
$$
A\ \tbD B-(\tbD^\ast A)\ B\ =\ 
\tbD ((\mathbf{T}^{-1}A)B)
$$
for any $A,B \in (\C^d)^\Z$.

As we have already done in previous subsections, to the shift operator $\mathbf{T}$ acting on sequences we may naturally associate a family of shifts on functions $\mathbf{T}_k(v)=v(\cdot+k)$. For concision's sake, from now on we will likewise denote by $(P(\mathbf{T}))_k$ the operator $P(\mathbf{T}_k)$. In particular, we will freely use notation $\bD_k$, $\tbD_k$, $\bfdelta_k$...

Now, our "non-degeneracy" assumption on the set of periodic traveling waves takes the following form
\be
\label{A_hamilton}
\begin{array}{r}
\textrm{Periodic traveling waves of \eqref{disHam} --- identified when coinciding up to time translation }\\ 
\textrm{ --- form a smooth manifold of dimension $d+2$, regularly parametrized by $(k,\bM, E)$,}\\ 
\textrm{ their wavenumber, the average-values of their components and their Hamiltonian.}
\end{array}\tag{A}
\ee
More explicitely
$$
E\ =\ \int_0^1 \cH (\uu^{(k,M,E)}(\zeta),\tbD_k\,\uu^{(k,M,E)}(\zeta))\,\dD\zeta\,.
$$

With notational convention similar to the one of previous subsections, a two-scale \emph{ansatz} leads, under a suitable normalization of parametrization, to
$$
u_0(X,T,\Theta)\ =\ \uu^{(k_0,\bM_0,E_0)(X,T)}\,(\Theta)
$$
with
$$
k_0 (X,T) := \d_X \phi_0(X,T)\,,\qquad\omega_0(X,T):= \d_T \phi_0(X,T)\,,
$$
$$
\bM_0(X,T)\ :=\ \int_0^1u_0(X,T,\Theta)\,\dD\Theta\,,   
$$
and
$$
E_0(X,T)\ :=\ \int_0^1 \cH (u_0(X,T,\Theta),\tbD_{k_0(X,T)}\,u_0(X,T,\Theta))\,\dD\Theta\,.
$$
Moreover we still have a law of conservation of waves
$$
\d_T\,k_0\ -\ \d_X(\omega(k_0,\bM_0))\ =\ 0
$$
and a conservation law for the averaged-values
$$
\d_T \bM_0\ =\ \bB\, \d_X F(k_0,\bM_0,E_0),  
$$
where 
$$ 
F(k,\bM,E)\ =\ \eta\int_0^1 \bfdelta_{k}(\cH)[\uu^{k,\bM,E}](\zeta)\,\dD\zeta
\ =\ \eta\int_0^1 \nabla_U\cH(\uu^{k,\bM,E}(\zeta),\tbD_k\uu^{k,\bM,E}(\zeta))\,\dD\zeta\,.
$$
The latter fact is easily deduced from the fact that by collecting terms of power $\epsilon^1$ stemming from the insertion of our \emph{ansatz} in our system we receive an equation in the form
$$
-\cL u_1\ +\ \d_T(u_0)\ - \tfrac{\eta}{2}(\mathbf{T}_{k_0}+\mathbf{T}_{-k_0})\bB\d_X(\bfdelta_{k_0}(\cH)[u_0])
\ =\ \d_\Theta(\cdots)+\bD_{k_0}(\cdots)\,.
$$

The derivation of an equation for the time evolution of $E_0$ requires more algebra. To write computations in a compact way we introduce some more pieces of notation. First we set $\cL_\xi:=\eD^{-\iD\xi\cdot}\,\cL\,\eD^{\iD\xi\cdot}$ then we denote $\cL^{(1)}$ and $\cL^{(2)}$ the operators involved in the expansion
$$
\cL_\xi\ =\ \cL\ +\ (\iD\xi\,k)\,\cL^{(1)}\ +\ (\iD\xi\,k)^2\,\cL^{(2)}\ +\ \cO(|\xi|^3)\,.
$$
With these preliminaries the equation obtained from the collection of coefficients of $\epsilon^1$ is explicitly written as
$$
\begin{array}{rcl}
-\cL u_1&+&\d_T(u_0)\ +\ (\d_T\phi_1+c(k_0,M_0,E_0)\d_X\phi_1)\,\d_\Theta u_0\\[1em]
&=&\d_X\phi_1\,\cL^{(1)}\d_\Theta u_0\ +\ \d_X^2\phi_0\,\cL^{(2)}\d_\Theta u_0
\ +\ (\cL^{(1)}-c(k_0,M_0,E_0))\,\d_X u_0\,.
\end{array}
$$
Since $\delta_{k_0}[\cH](u_0)$ belongs to the kernel of $\cL^*$, we shall average over $\Theta\in(0,1)$ the scalar product of the previous system with $\delta_{k_0}[\cH](u_0)$. Our claim is that we receive
$$
\d_TE_0\ -\ \d_X(S(k_0,\bM_0,E_0))\ =\ 0
$$
where
$$
\begin{array}{rcl}
\displaystyle
S(k,\bM,E)&=&\displaystyle
\eta\int^1_0 \tfrac12\, \mathbf{T}_{-k}(\bfdelta_k(\cH)[\uu^{(k,M,E)}])\,\bB\, \bfdelta_k(\cH)[\uu^{(,M,E)}]\\[1em]
&+&\displaystyle
\eta\int^1_0\mathbf{T}_{-k}(\nabla_{\tbD\,U}\cH(\uu^{(k,M,E)},\tbD_k\,\uu^{(k,M,E)}))\cdot \bJ_k\, \bfdelta_k(\cH)[\uu^{(k,M,E)}]
\end{array}
$$

Indeed it does follow from direct computations. One one hand, when $a\in\{\zeta,k,M_1,\cdots,M_d,E\}$, one computes that 
\be\label{e:d-Ham}
\delta_{k}[\cH](\uu)\cdot \d_a\uu\ =\ \d_a(\cH(\uu,\tbD_k\uu))-\eta \mathbf{T}_{-k}(\nabla_{\tbD\,U}\cH(\uu,\tbD_k\,\uu))\cdot\d_\zeta \uu\ \d_ak
\ +\ (\cdots)
\ee
where $(\cdots)$ denotes mean-free terms and dependence of $\uu$ on parameters has been omitted. On the other hand, for $a\in\{\zeta,k,M_1,\cdots,M_d,E\}$, we also have that
\be\label{e:d-Hamflux}
\begin{array}{rcl}
\displaystyle
\delta_{k}[\cH](\uu)&\cdot& \displaystyle
\left((\cL^{(1)}-\uc)\d_a\uu\ +\ \,\cL^{(2)}(\d_\zeta \uu)\d_ak\right)\\[1em]
&=&\displaystyle
\d_a\left(\eta \tfrac12\, \mathbf{T}_{-k}(\bfdelta_k(\cH)[\uu])\,\bB\, \bfdelta_k(\cH)[\uu]
\ +\ \eta \mathbf{T}_{-k}(\nabla_{\tbD\,U}\cH(\uu,\tbD_k\,\uu))\cdot \bJ_k\, \bfdelta_k(\cH)[\uu]\right)\\[1em]
&-&\displaystyle
\eta \mathbf{T}_{-k}(\nabla_{\tbD\,U}\cH(\uu,\tbD_k\,\uu))\cdot\d_\zeta \uu\ \d_a\omega
\ +\ (\cdots)
\end{array}
\ee
with the same convention. This yields our claim by integration using that $\d_Tk_0=\d_X(\omega(k_0,\bM_0,E_0))$. Alternatively one could have derived this last equation by directly expanding and averaging equation \eqref{disConLoc}.


Altogether we have derived the following averaged equations for the slow modulation of local parameters
\begin{equation}
\label{WhithamHAM}
\left\{
\begin{array}{rcl}
\d_T k &=&\d_X\omega(k,\bM,E)\\[0.5em]
\d_T\bM &=&\bB\d_XF(k,\bM,E) \\[0.5em]
\d_T E &=&\d_X S(k,\bM,E)\,.
\end{array}\right.
\end{equation}


\section{Analytic framework}

We collect in this section some preliminaries from functional analysis that will allow us to formulate in which sense the formally derived averaged equations provide some valuable pieces of information at the spectral level.

\subsection{Integral transform}\label{s:Bloch}

To analyze linearized lattice systems with periodic coefficients we shall make use of an adapted integral transform, usually called discrete Bloch transform. Though probably less known that its continuous counterpart it is a natural functional tool that have already received various applications in related contexts, see for instance \cite{ChirilusBruckner-Chong-Prill-Schneider_wave-packets_periodic-chains_modulation,ChirilusBruckner-Schneider_standing-pulse_periodic-media}.

\subsubsection{Discrete Bloch Transform}\label{s:DBT}

Let $N\in\N^*$ be given. We define the following transform 
\be
\cB:\quad \ell^2(\Z;\C^d)\ \rightarrow\ L^2\left( \intl0,N-1\intr\times\left[-\tfrac{\pi}{N},\tfrac{\pi}{N}\right];\C^d\right)\,,\quad
f\ \mapsto\ \check{f}
\ee
where for $j\in \intl0,N-1\intr$
\be\label{DBT}
\check{f}(j,\cdot):= \lim_{k_0\to\infty}\sum\limits_{k\in\intl-k_0,k_0\intr}\eD^{-\iD(kN+j)\ \cdot}\ f_{kN+j}\,,
\ee
the limit being taken in the $L^2\left(\left[-\tfrac{\pi}{N},\tfrac{\pi}{N}\right] \right)$-sense.
Up to a multiplicative constant this is a total isometry whose reciprocal is given as
\be
\cB^{-1}:\quad L^2\left( \intl0,N-1\intr\times\left[-\tfrac{\pi}{N},\tfrac{\pi}{N}\right];\C^d\right)\ \rightarrow\ \ell^2(\Z;\C^d)\,,\quad
\check{f}\ \mapsto\ f
\ee
where for $(k,j)\in \Z\times\intl0,N-1\intr$
\be\label{IDBT}
f_{kN+j}\ =\ \dfrac{N}{2\pi}\,\int_{-\tfrac{\pi}{N}}^{\tfrac{\pi}{N}}\eD^{\iD\,(kN+j)\,\xi}\ \check{f}(j,\xi)\,\dD\xi\,.
\ee

From now on we'll constantly identify functions on $\Z/(N\Z)$ with functions on $\intl0,N-1\intr$ through usual restriction/periodic extension processes. To this purpose we stress that formula~\eqref{DBT} already provides the needed periodic extension. With this identification, inverse formula~\eqref{IDBT} simply reads for any $j\in\Z$
\be
f_{j}\ =\ \dfrac{N}{2\pi}\,\int_{-\tfrac{\pi}{N}}^{\tfrac{\pi}{N}}\eD^{\iD\,j\,\xi}\ \check{f}(j,\xi)\,\dD\xi\,.
\ee

\subsubsection{Discrete Bloch symbols}\label{s:DBT}

To any element $\xi\mapsto A_\xi$ of $L^\infty(\left[-\tfrac{\pi}{N},\tfrac{\pi}{N}\right];\cL(\ell^2(\Z/(N\Z);\C^d)))$ one may associate a bounded operator $A\in\cL(\ell^2(\Z;\C^d))$ such that, for any $f\in\ell^2(\Z;\C^d)$, 
$$
(Af)\check{\ }(\cdot,\xi)\ =\ A_\xi \check{f}(\cdot,\xi)\,.
$$
Explicitly 
\be
(Af)_{j}\ =\ \dfrac{N}{2\pi}\,\int_{-\tfrac{\pi}{N}}^{\tfrac{\pi}{N}}\eD^{\iD\,j\,\xi}\ (A_\xi\,\check{f}(\cdot,\xi))_j\,\dD\xi\,.
\ee
Moreover up to a multiplicative constant the map $$L^\infty(\left[-\tfrac{\pi}{N},\tfrac{\pi}{N}\right];\cL(\ell^2(\Z/(N\Z);\C^d)))\to\cL(\ell^2(\Z;\C^d))$$ is a partial isometry. From this follows the classical identity between the spectrum of the synthetized operator $A$ and the essential range of $\xi\mapsto \sigma(A_\xi)$. In particular, in the case where $\xi\mapsto A_\xi$ belongs to $\cC^0(\left[-\tfrac{\pi}{N},\tfrac{\pi}{N}\right];\cL(\ell^2(\Z/(N\Z);\C^d)))$, this yields
$$
\sigma(A)\ =\ \bigcup_{\xi\in\left[-\tfrac{\pi}{N},\tfrac{\pi}{N}\right]} \sigma(A_\xi)
$$
so that the spectrum of $A$ may be analyzed by looking at the family of spectra of the finite-dimensional symbols $A_\xi$.

Reciprocally for a large class of $A\in\cL(\ell^2(\Z;\C^d))$ one may define corresponding symbols $\xi\mapsto A_\xi$ formally given by
$$
A_\xi\ =\ \eD^{-\iD\,\cdot\,\xi}\ A\ \eD^{\iD\,\cdot\,\xi}
$$
that is, for $f\in\ell^2(\Z/(N\Z);\C^d)$
$$
(A_\xi f)_j\ =\ \eD^{-\iD\,j\,\xi}\ (A\ (\eD^{\iD\,\cdot\,\xi}f))_j
$$
with an extended sense of $A$. To make things more concrete, we focus now on the case where $A$ is an $N$-periodic discrete differential operator in the sense that it is given as 
$$
A\ =\ a(\cdot,\mathbf{T})\quad\textrm{with}\quad a(\cdot,X)\ =\ \sum_{j\in\Z}\, a_j(\cdot)\,X^j
$$
where $\mathbf{T}$ is again the left shift operator and $a:\Z\to\cL(\C^d)$ is $N$-periodic. 
The Bloch symbols of $A$ are then given by
$$
A_\xi\ =\ a(\cdot,\eD^{\iD\,\xi}\,\mathbf{T})
$$
where $\mathbf{T}$ is now the left shift operator acting on $(\C^d)^{\Z/(N\Z)}$ and $a$ is considered as a function on $\Z/(N\Z)$.

\subsection{Linear evolution}\label{s:spectrum}


By linearizing original equations we receive equations in the form 
$$
\dfrac{\dD}{\dD t}V (t)\ =\ A(t)\,V(t)
$$
where $A(\cdot)$ is a $1/\uom$-periodic family of bounded operators generating $S(\cdot,\cdot)$ an evolution system on $\ell^2(\Z;\C^d)$ (in the sense defined for instance in \cite[Chapter 5]{Pazy}). Explicitly we have
$$
(\forall s\in\R,\ S(s,s)\ =\ \Id)\quad\textrm{and}\quad (\forall (s,r,t)\in\R^3,\ S(t,s)\ =\ S(t,r)S(r,s))
$$
and
$$
\forall (s,t)\in\R^2,\ (\d_tS)(t,s)\ -\ A(t)\,S(t,s)\ =\ 0\quad\textrm{and}\quad(\d_sS)(t,s)\ +\ S(t,s)\,A(s)\ =\ 0\,.
$$
Mark that periodicity implies that linear stability in the sense that
$$
\sup_{t\in\R_+}\ \|S(t,0)\|\ <\ \infty
$$
is equivalent to
$$
\sup_{k\in\N}\ \|S(1/\uom,0)^k\|\ <\ \infty\,.
$$
Likewise linear exponential instability in the sense that
$$
\exists \gamma>0\,,\ \inf_{t\in\R_+}\ \eD^{-\gamma\,t}\,\|S(t,0)\|\ >\ 0
$$
is equivalent to
$$
\exists \kappa>1\,,\ \inf_{k\in\N}\ \kappa^{-k}\,\|S(1/\uom,0)^k\|\ >\ 0\,,
$$
that is to $\rho(S(1/\uom,0))>1$.

\br\label{space-translation-invariance} Here we use time-translation invariance and periodicity in time as a substitute to space-translation invariance and stationarity in a moving frame that play a crucial role in similar analyses of continuous systems. However as pointed out for instance in \cite{Friesecke-Pego-I,Friesecke-Pego-II,Friesecke-Pego-III,Friesecke-Pego-IV} the loss of an exact invariance in space does not preclude arguments based on a pseudo-invariance. Notably, mark that for all $t\in\R^+$
$$
A(t+1/c)=\mathbf{T}^{-1}A(t)\mathbf{T}\,.
$$
This implies that for all $(t,s)\in\R$
$$
S(t+1/c,s+1/c)=\mathbf{T}^{-1}S(t,s)\mathbf{T}\,.
$$
As a consequence, linear stability is also equivalent to
$$
\sup_{k\in\N}\|(\mathbf{T}S(1/c,0))^k\|<\infty\,.
$$
The fact that such an alternative exist is obviously crucial when dealing with fronts, shocks or solitary waves rather than with periodic waves.
\er

From now on, in order to make the most of the discrete Bloch transform we focus on cases where, for some $N\in\N^*$, $A(\cdot)$ is valued in $N$-periodic discrete differential operators that is we restrict to cases where linearization is about a periodic traveling wave with an integer period, $N$. Then $A_\xi(\cdot)$ the value at $\xi$ of the corresponding family of Bloch symbols generates the evolution system given by the value at $\xi$, $S_\xi(\cdot,\cdot)$, of the family of Bloch symbols associated with $S(\cdot,\cdot)$. In our context some continuity in the $\xi$-variable is available and linear stability reads
$$
\sup_{\xi\in\left[-\tfrac{\pi}{N},\tfrac{\pi}{N}\right]}\sup_{k\in\N}\ \|S_\xi(1/\uom,0)^k\|\ <\ \infty\,,
$$
while linear exponential instability may be written as 
$$
\exists \xi\in\left[-\tfrac{\pi}{N},\tfrac{\pi}{N}\right]\,,\ \rho(S_\xi(1/\omega,0))>1\,.
$$
This implies that a necessary condition for linear stability is that for any $\xi\in\left[-\tfrac{\pi}{N},\tfrac{\pi}{N}\right]$
$$
\sigma\left(S_\xi(1/\uom,0)\right)\ \subset\ \bar B(0,1)\,.
$$

The fact that each periodic wave is embedded in a family of neighboring waves implies that $1$ always belongs to the spectrum of $S_0(1/\uom,0)$. It is then crucial in the analysis of linear stability to determine how the eigenvalue $1$ evolves when $\xi$ is varied. Our goal is to prove that this piece of information is precisely given by the averaged equations.

\subsection{Algebraic interplay}\label{s:algebraic}

To analyze the evolution of the spectrum of $S_\xi(1/\omega,0)$ when $\xi$ is small, we first stress the following classical identity, for any $\xi\in\left[-\tfrac{\pi}{N},\tfrac{\pi}{N}\right]$
\be\label{solFor}
\forall (s,t)\in\R^2,\quad S_\xi(t,s)\ =\ S_0(t,s)\ +\ \int_s^tS_0(t,r)\,(A_\xi(r)-A_0(r))\,S_\xi(r,s)\,\dD r\,.
\ee
In particular, expanding $A_\xi(\cdot)$ as
$$
A_\xi(\cdot) \ \stackrel{\xi\to0}{=}\ A_0(\cdot)\,+\,\iD\,\xi\ A^{(1)}(\cdot)\,+\,(\iD\,\xi)^2\ A^{(2)}(\cdot)\,+\,\cO(|\xi|^3)\,,
$$
uniformly in time, for some $A^{(1)}(\cdot)$, $A^{(2)}(\cdot)$, we receive locally uniformly in time
$$
S_\xi(t,s) \ \stackrel{\xi\to0}{=}\ S_0(t,s)\,+\,\iD\,\xi\ S^{(1)}(t,s)\,+\,(\iD\,\xi)^2\ S^{(2)}(t,s)\,+\,\cO(|\xi|^3)\,,
$$
with
\be\label{e:S-expand}
\begin{array}{rcl}
S^{(1)}(t,s)&=&\displaystyle
\int_s^tS_0(t,r)\,A^{(1)}(r)\,S_0(r,s)\,\dD r\\[1em]
S^{(2)}(t,s)&=&\displaystyle
\int_s^tS_0(t,r)\,A^{(2)}(r)\,S_0(r,s)\,\dD r\ +\ \int_s^tS_0(t,r)\,A^{(1)}(r)\,S^{(1)}(r,s)\,\dD r\,.
\end{array}
\ee

We have already encountered a similar expansion for $\cL_\xi:=\eD^{-\iD\xi\cdot}\,\cL\,\eD^{\iD\xi\cdot}$. Operators $\cL_\xi$ acting on functions of period $1$ are actually continuous Bloch symbols of the operator obtained by extending $\cL$ to functions defined on the full line. Since we won't need this relation here, we mostly use the fact that by expanding
$$
\cL_\xi\ =\ \cL\ +\ (\iD\xi\,k)\,\cL^{(1)}\ +\ (\iD\xi\,k)^2\,\cL^{(2)}\ +\ \cO(|\xi|^3)\,.
$$
we obtain operators $\cL^{(j)}$ that are practical shorthands to denote relations obtained by differentiating profile equation with respect to parameters. For instance, when $a$ denotes a parameter which is not $k$ (but possibly $\zeta$, to account for phase shifts), one receives
\be\label{e:a-profile}
\cL\d_a\uu\ =\ \d_a\omega\d_\zeta\uu
\ee
while
\be\label{e:k-profile}
\cL\d_k\uu\ =\ -\uk\d_kc\,\d_\zeta\uu\ -\ \cL^{(1)}\d_\zeta\uu\,.
\ee

Proofs of our main results strongly rely on the interplay between $\cL_\xi$ and $A_\xi$ that allows to use the previous relations. Explicitly to any function $v:\R\to\C^d$ of period $1$, we associate $V^v:\R\to (\C^d)^\Z$ by
$$
V^v(j,t)\ =\ v(\uk\,(j-\uc t))\,.
$$
Observe that when $v$ and $w$ are one-periodic
$$
\int_0^{1/\uom}\langle V^w(s),V^v(s)\rangle\dD s\ =\ \frac{1}{\uk\uom}\langle w,v\rangle\,.
$$
Note also that when $v$ is one-periodic $V^v(\cdot,t)$ is $1/\uom$-periodic in time and valued in $N$-periodic sequences. In this case,
\be\label{e:convert}
\begin{array}{rcccccl}
\frac{\dD}{\dD t}V^v(t)&-&A_\xi(t)\,V^v(t)&=&-V^{\cL_\xi\,v}(t)&-&\iD\xi\uom V^v(t)\\[0.5em]
\frac{\dD}{\dD t}V^v(t)&+&A_\xi(t)^*\,V^v(t)&=&V^{\cL_\xi^*\,v}(t)&-&\iD\xi\uom V^v(t)\,.
\end{array}
\ee
In particular
\be\label{e:dual-convert}
\frac{\dD}{\dD t}(\langle V^w, V^v\rangle)(t)
\ =\ \langle V^{\cL_\xi^*\,w}(t),V^v(t)\rangle\,-\,\langle V^w(t),V^{\cL_\xi\,v}(t)\rangle\,,
\ee
which turns to be crucial to transport duality relations. Obviously by expanding in $\xi$ the foregoing relations we obtain corresponding relations between operators $A^{(j)}$ and operators $\cL^{(j)}$, for instance
\be\label{e:expand-convert}
\begin{array}{rcccl}
\frac{\dD}{\dD t}V^v(t)&-&A_0(t)\,V^v(t)&=&-V^{\cL\,v}(t)\\[0.5em]
&-&A^{(1)}(t)\,V^v(t)&=&-\uk\,V^{(\cL^{(1)}-\uc)\,v}(t)\\[0.5em]
&-&A^{(2)}(t)\,V^v(t)&=&-\uk^2\,V^{\cL^{(2)}\,v}(t)\\[0.5em]
\frac{\dD}{\dD t}V^v(t)&+&A_0(t)^*\,V^v(t)&=&V^{\cL^*\,v}(t)\,.
\end{array}
\ee

Mark that from \eqref{e:a-profile} and \eqref{e:convert} 
$$
\forall (t,s)\in\R^2\,,\ V^{\d_\zeta\uu}(t)\ =\ S_0(t,s)\,V^{\d_\zeta\uu}(s)\,.
$$
In particular, $V^{\d_\zeta\uu}(0)\in\ker (S_0(1/\uom,0)-\Id)$. More generally \eqref{e:a-profile} and \eqref{e:convert} provide as many members of the generalized kernel of $S_0(1/\uom,0)-\Id$ as the expected dimension of the family of periodic traveling waves (counted up to translation). Regardless of the version of the equations that we are considering, assumption A may be reformulated as the fact that these elements form a basis of this generalized kernel. Moreover \eqref{e:k-profile}, \eqref{e:convert} and \eqref{e:S-expand} yield
\be\label{e:k-S}
\forall (t,s)\in\R^2\,,\ V^{\d_k\uu}(t)\ =\ S_0(t,s)\,V^{\d_k\uu}(s)\ -\ (t-s)\d_k\omega\, V^{\d_\zeta\uu}(s)
\ +\ \frac1\uk S^{(1)}(t,s)V^{\d_\zeta\uu}(s)\,.
\ee


\section{Spectral validation of averaged equations}\label{s:spectrum}

\subsection{Reaction-diffusion case}\label{s:spectrum-RD}

Firstly we analyze the linearization of \eqref{nonlLA} about $\uU(\cdot)$ the periodic traveling wave associated to $\uu^{\uk}$, with $\uk=1/N$. The linearized evolution obeys $(\forall t\,,\quad \tfrac{\dD}{\dD t}V (t)\ =\ A(t)\,V(t))$ where
$$
A(\cdot)\ =\ \mu\,(\mathbf{T}-2\Id+\mathbf{T}^{-1}))\,+\,\dD f(\uU(\cdot))
$$
with usual identification of functions with their local action on sequences, that is, here, 
$$
\forall (t,j)\,,\ [\dD f(\uU(t))V]_j\ =\ \dD f(\uU_j(t))\,V_j\,.
$$
Explicitly here
$$
A^{(1)}(\cdot)\equiv \mu\,(\mathbf{T}-\mathbf{T}^{-1})
\quad\textrm{and}\quad
A^{(2)}(\cdot)\equiv \frac12\mu\,(\mathbf{T}+\mathbf{T}^{-1})\,.
$$

Now our main goal is to prove the following result that validates \eqref{W_RD} and \eqref{W2_RD} at the spectral level. 

\begin{theorem}\label{propRD}
Assuming that condition \ref{A_RD} holds in a neighborhood of $\uu^{\uk}$, there exist positive $\epsilon_0$ and $\xi_0\in(0,\pi/N)$ and an analytic curve $\lambda:[-\xi_0,\xi_0]\to\C$ such that
$$
\forall \xi\in[-\xi_0,\xi_0]\,,\ \sigma(S_\xi(1/\omega^{\uk},0))\,\cap\,B(1,\epsilon_0)\ =\ \{\lambda(\xi)\}\,.
$$
Moreover, for $\xi\in[-\xi_0,\xi_0]$, eigenvalue $\lambda(\xi)$ of $S_\xi(1/\omega^{\uk},0)$ is simple and
\be\label{eig_expand_RD}
\lambda(\xi)\ \stackrel{\xi\to0}{=}\ \eD^{\frac{1}{\omega^{\uk}}\,\left(\iD\uk\xi\d_k\omega^{\uk}\,+\,(\iD\uk\xi)^2\,d(\uk)\right)}\ +\ \cO(|\xi|^3)\,,
\ee
where $d(\uk)$ is given by \eqref{linvisc} under normalization \ref{normal_RD}. 
\end{theorem}

The above theorem has many direct implications. For instance, we readily see that $d(\uk)>0$ yields a linear exponential instability caused by side-band perturbations, that is, by perturbations with arbitrary small --- but non zero --- spatial Floquet exponents. It also proves as expected that the single linear group velocity is indeed $\d_k\omega^{\uk}$.

\begin{proof}
Assumption \ref{A_RD} implies that $1$ is a simple eigenvalue of $S_0(1/\uom,0)$ so that the only point that does not follow from classical results on regular perturbations of simple eigenvalues is expansion \eqref{eig_expand_RD}. To identify coefficients we expand 
$$
\lambda(\xi) \ \stackrel{\xi\to0}{=}\ 1\,+\,\iD\,\xi\ \lambda^{(1)}\,+\,(\iD\,\xi)^2\ \lambda^{(2)}\,+\,\cO(|\xi|^3)\,,
$$
and a corresponding eigenvector $\varphi_\xi$, which also depends analytically on $\xi$, as
$$
\varphi_\xi\ \stackrel{\xi\to0}{=}\ V^{\d_\zeta\uu}(0)\,+\,\iD\xi\ \varphi^{(1)}\,+\,(\iD\xi)^2\ \varphi^{(2)}\,+\,\cO(|\xi|^3)\,.
$$
Setting $\uup=1/\omega^{\uk}$ for writing convenience, we now expand relation $S_\xi(\uup,0)\varphi_\xi=\lambda(\xi)\varphi_\xi$. 

At first order in $\iD\xi$ we receive
\be\label{first_order_RD}
(S_0(\uup,0)-\Id)\,\varphi^{(1)}\ =\ \lambda^{(1)}\ V^{\d_\zeta\uu}(0)-S^{(1)}(\uup,0)V^{\d_\zeta\uu}(0)\,.
\ee
Hence
$$
(S_0(\uup,0)-\Id)\,(\varphi^{(1)}-\uk V^{\d_k\uu}(0))\ =\ \left(\lambda^{(1)}-\uk\uup\d_k\omega^{\uk}\right)\ V^{\d_\zeta\uu}(0)\,.
$$
Since $1$  is a simple eigenvalue of $S_0(\uup,0)$, $V^{\d_\zeta\uu}(0)$ does not belong to the image of $S_0(\uup,0)-\Id$. From this we infer the desired 
$$
\lambda^{(1)}\ =\ \uk\uup\d_k\omega^{\uk}
\qquad\textrm{and}\qquad
\varphi^{(1)}-\uk V^{\d_k\uu}(0)\in \ker (S_0(\uup,0)-\Id)=\C V^{\d_\zeta\uu}(0)\,.
$$

By expanding now at the second order in $\iD\xi$ we obtain
\be\label{second_oder_RD}
(S_0(\uup,0)-\Id)\,\varphi^{(2)}\ =\ \lambda^{(2)}\ V^{\d_\zeta\uu}(0)
+\uk\uup\d_k\omega^{\uk}\,\varphi^{(1)}
-S^{(1)}(\uup,0)\varphi^{(1)}\,-\,S^{(2)}(\uup,0)V^{\d_\zeta\uu}(0)\,.
\ee
To proceed we also need to involve $\uad_\uk$ defined in Subsection~\ref{s:Whitham-RD}, that is $\uad_\uk$ spans $\ker\cL^*$ and $\langle\uad_\uk,\d_\zeta\uu^{\uk}\rangle=1$ (in the $L_{per}^2([0,1])$ sense). From \eqref{e:convert} we deduce that 
$$
\forall (t,s)\,,\quad S_0(s,t)^*\,\Vad(t)\ =\ \Vad(s)\,,
$$
thus that $V^{\uad_\uk}(0)\in\ker (S_0(\uup,0)^*-\Id)$ and from \eqref{e:dual-convert} that
$$
\forall t\,,\ \langle V^{\uad_\uk}(t),V^{\d_\zeta\uu}(t)\rangle
\ =\ \frac1\uup\int_0^{\uup}\langle V^{\uad_\uk}(s),V^{\d_\zeta\uu}(s)\rangle\,\dD s
\ =\ \frac1\uk\langle\uad_\uk,\d_\zeta\uu^{\uk}\rangle\ =\ \frac1\uk\,.
$$
Since $-\uk\uup\d_k\omega^{\uk}\,V^{\d_\zeta\uu}(0)+S^{(1)}(\uup,0)V^{\d_\zeta\uu}(0)$ belongs to the range of $S_0(\uup,0)-\Id$ this yields
$$
\tfrac1\uk\lambda^{(2)}\ =\ \langle V^{\uad_\uk}(0),-\uk\uup\d_k\omega^{\uk}\,V^{\d_k\uu}(0)
\,+\,\uk S^{(1)}(\uup,0)V^{\d_k\uu}(0)\,+\,S^{(2)}(\uup,0)V^{\d_\zeta\uu}(0)\rangle\,.
$$
Now with \eqref{e:S-expand} we compute
$$
\begin{array}{rcl}
\langle V^{\uad_\uk}(0)&\hspace{-1em},&\hspace{-1em}\uk S^{(1)}(\uup,0)V^{\d_k\uu}(0)\,+\,S^{(2)}(\uup,0)V^{\d_\zeta\uu}(0)\rangle\\[1em]
&=&\displaystyle
\int_0^\uup \langle V^{\uad_\uk}(s),A^{(1)}(s)[\uk S_0(s,0)V^{\d_k\uu}(0)\,+\,S^{(1)}(s,0)V^{\d_\zeta\uu}(0)]\,+\,A^{(2)}(s)V^{\d_\zeta\uu}(s)\rangle\ \dD s\\[1em]
&=&\displaystyle
\int_0^\uup \langle V^{\uad_\uk}(s),A^{(1)}(s)[\uk\,s\,\d_k\omega^{\uk}\ V^{\d_\zeta\uu}(s)\,+\,\uk V^{\d_k\uu}(s)]\,+\,A^{(2)}(s)V^{\d_\zeta\uu}(s)\rangle\ \dD s\\[1em]
&=&\displaystyle
\int_0^\uup \langle V^{\uad_\uk}(s),\uk^2\,s\,\d_k\omega^{\uk}\ V^{(\cL^{(1)}-\uc)\d_\zeta\uu}(s)\,+\,\uk^2 V^{(\cL^{(1)}-\uc)\d_k\uu}(s)]\,+\,\uk^2V^{\cL^{(2)}\d_\zeta\uu}(s)\rangle\ \dD s\,.
\end{array}
$$
Now, since $(\cL^{(1)}-\uc)\d_\zeta \uu=\d_k\omega\d_\zeta\uu-\cL\d_k\uu$, by using \eqref{e:expand-convert} and normalization \ref{normal_RD} and integrating by part, we observe that
$$
\int_0^\uup \langle V^{\uad_\uk}(s),\uk^2\,s\,\d_k\omega^{\uk}\ V^{(\cL^{(1)}-\uc)\d_\zeta\uu}(s)\rangle\ \dD s\ =\ 
\frac1\uk\ \left[\frac12\uk^2\uup^2(\d_k\omega^{\uk})^2+\uk^2\uup\d_k\omega^{\uk}\,\langle V^{\uad_\uk}(0),V^{\d_k\uu}(0)\rangle\right]
$$
which yields
$$
\lambda^{(2)}\ =\ \uup\uk^2\,d(\uk)\,+\,\tfrac12\,(\uup\uk\d_k\omega^{\uk})^2
$$
as desired.
\end{proof}

\subsection{General mixed case}

Now our main goal is to unravel the role of \eqref{WhithamCON} in the spectral analysis of the linearization of \eqref{disCON}. 

\begin{theorem}\label{propMIX}
Assuming that condition \ref{A_mixed} holds in a neighborhood of $\uu^{\uk,\ubM}$, there exist positive $\epsilon_0$ and $\xi_0\in(0,\pi/N)$ and continuous curves $\lambda_\alpha:[-\xi_0,\xi_0]\to\C$,  $\alpha =1,\cdots, d_1+1$, such that
$$
\forall \xi\in[-\xi_0,\xi_0]\,,\ \sigma(S_\xi(1/\omega^{\uk,\ubM},0))\,\cap\,B(1,\epsilon_0)\ =\ \{\lambda_1(\xi),\cdots,\lambda_{d_1+1}(\xi)\}\,.
$$
Moreover, for $\xi\in[-\xi_0,\xi_0]$ and $\alpha =1,\cdots, d_1+1$
\be\label{eig_expand_MIX}
\lambda_\alpha(\xi)\ \stackrel{\xi\to0}{=}\ \eD^{\frac{1}{\omega^{\uk,\ubM}}\,\iD\uk\xi\,a_\alpha}\ +\ \cO(|\xi|^2)\,,
\ee
where $a_1,\cdots,a_{d_1+1}$ are the characteristic speeds of the Whitham system \eqref{WhithamCON} linearized about parameters $(\uk,\ubM)$.
\end{theorem}

In contrast with what happens in the reaction-diffusion case, that is averaged to a scalar equation, the above first-order expansion already provides an instability criterion. If the Whitham system \eqref{WhithamCON} is not weakly hyperbolic at $(\uk,\ubM)$, that is, if some of the $a_\alpha$ are not real, then the corresponding wave is linearly exponentially instable to side-band perturbations. In contrast, when weak hyperbolicity is met, then the above theorem proves that those $a_\alpha$ do give the $d_1+1$ linear group velocities, a non trivial fact.

\begin{proof}
Let us denote by $e_1,\cdots,e_{d_1}$ both the $d_1$ first members of the canonical basis of $\C^d$ and the constant functions with corresponding values. As a consequence of assumption \ref{A_mixed}, there exists $\uad$ belonging to the generalized kernel of $\cL^*$, orthogonal to $\d_{\bM_1}\uu,\cdots,\d_{\bM_{d_1}}\uu$ and such that $\langle\uad,\d_\zeta\uu\rangle=1$. Moreover, by appealing to Subsection~\ref{s:algebraic} as in the reaction-diffusion case, one obtains that $(V^{\d_\zeta\uu},V^{\d_{\bM_1}\uu},\cdots,V^{\d_{\bM_{d_1}}\uu})$ and $(\uk V^{\uad},\uk V^{e_1},\cdots,\uk V^{e_{d_1}})$ form dual bases of the generalized kernels of $S_0(1/\omega^{\uk,\ubM},0)-\Id$ and $S_0(1/\omega^{\uk,\ubM},0)^*-\Id$.

By Kato's perturbation method \cite[pp.~99-100]{Kato} we may continue these bases as dual bases $(q_0(\xi),\cdots,q_{d_1}(\xi))$ and $(\tq_0(\xi),\cdots,\tq_{d_1}(\xi))$ of the generalized eigenspaces of $S_\xi(1/\omega^{\uk,\ubM},0)$ and $S_\xi(1/\omega^{\uk,\ubM},0)^*$ associated with their spectra in $B(1,\epsilon_0)$ as long as $|\xi|\leq \xi_0$ provided $\xi_0$ and $\epsilon_0$ are small enough. Then, for $|\xi|\leq\xi_0$, the spectrum of $S_\xi(1/\omega^{\uk,\ubM},0)-\Id$ in $B(0,\epsilon_0)$ is the spectrum of the matrix
$$
\Omega_\xi\ =\ \left[\left\langle\tq_\alpha(\xi), \left(S_\xi(1/\omega^{\uk,\ubM},0)-\Id\right)q_\beta(\xi)\right\rangle\right]_{0\leq\alpha,\beta\leq d_1}\,.
$$
Observe that 
$$
\Omega_0\ =\ \begin{pmatrix}0&\uup\d_{\bM_1}\omega&\cdots&\uup\d_{\bM_{d_1}}\omega\\
\vdots&&0_{d_1\times d_1}&\\
0&&&\end{pmatrix}
$$
where $\uup=1/\omega^{\uk,\ubM}$. Expanding 
$$
\Omega_\xi\ =\ \Omega_0\,+\,\iD\uk\xi\,\Omega^{(1)}\,+\,(\iD\uk\xi)^2\,\Omega^{(2)}\,+\,\cO(|\xi|^3)\,,
$$
we also obtain since 
$$
\tfrac1\uk S^{(1)}(\uup,0)\,V^{\d_\zeta\uu}(0)\ =\ 
-(S_0(\uup,0)-\Id)\,V^{\d_k\uu}(0)+\uup\,\d_k\omega^{\uk}\ V^{\d_\zeta\uu}(0)
$$
that
$$
\Omega^{(1)}\ =\ \begin{pmatrix}*&\cdots&*\\
0&&\\
\vdots&*&\\
0&&
\end{pmatrix}\,.
$$
The upshot is that the matrix $\tilde\Omega_\xi$ defined by
$$
\tilde\Omega_\xi\ :=\ \frac{1}{\iD\uk\xi}\Sigma_\xi^{-1}\Omega_\xi\Sigma_\xi
\qquad\textrm{with}\qquad
\Sigma_\xi\ =\ 
\begin{pmatrix}(\iD\uk\xi)^{-1}&O_{1\times d_1}\\
0_{d_1\times 1}&\I_{d_1\times d_1}\end{pmatrix}
$$
depends analytically on $\xi$. The heuristics behind this transformation is that the change of basis transforms a phase-like coordinate to a wavenumber-like coordinate as in the derivation of the Whitham system \eqref{WhithamCON} while the division by $\iD\uk\xi$ changes eigenvalues in velocities.

Now the point is to identify $\tilde\Omega_0$. We already know that
$$
\tilde\Omega_0\ =\ \left(\begin{array}{c|c}(\Omega^{(1)})_{\alpha,\beta}&(\Omega_0)_{\alpha,\beta}\\\hline
(\Omega^{(2)})_{\alpha,\beta}&(\Omega^{(1)})_{\alpha,\beta}\end{array}\right)_{0\leq\alpha,\beta\leq d_1}\,.
$$
and, for $\beta=1,\cdots,d_1$, $(\Omega_0)_{1,\beta}=\uup\d_{\bM_\beta}\omega$. First 
$$
\begin{array}{rcl}
(\Omega^{(1)})_{0,0}&=&\displaystyle
\tfrac{1}{\iD\uk}\langle \tq_0(0),\left(S_0(\uup,0)-\Id\right)\d_\xi q_0(0)\rangle
\,+\,\tfrac{1}{\uk}\langle \tq_0(0),S^{(1)}(\uup,0) q_0(0)\rangle\\[1em]
&=&\displaystyle
\uup\d_k\omega\,+\,\tfrac{1}{\iD\uk}\langle \tq_0(0),\left(S_0(\uup,0)-\Id\right)(\d_\xi q_0(0)-\iD\uk V^{\d_k\uu}(0))\rangle\,.
\end{array}
$$
Our claim is that the last term vanishes. Indeed since $(S_0(\uup,0)-\Id)q_0(0)=0$ and, for $|\xi|\leq\xi_0$,
$$
(S_\xi(\uup,0)-\Id)q_0(\xi)\ \in\ \textrm{Span}(\{q_0(\xi),\cdots,q_{d_1}(\xi)\})
$$
we conclude
$$
\tfrac{1}{\iD\uk}(S_0(\uup,0)-\Id)\d_\xi q_0(0)
\,+\,\tfrac1\uk S^{(1)}(\uup,0)q_0(0)\ \in\ \textrm{Span}(\{q_0(0),\cdots,q_{d_1}(0)\})
$$
hence
$$
\left(S_0(\uup,0)-\Id\right)(\d_\xi q_0(0)-\iD\uk V^{\d_k\uu}(0))
\in\ \textrm{Span}(\{q_0(0),\cdots,q_{d_1}(0)\})\,.
$$
This implies
$$
\d_\xi q_0(0)-\iD\uk V^{\d_k\uu}(0)
\in\ \textrm{Span}(\{q_0(0),\cdots,q_{d_1}(0)\})
$$
which in turn yields 
$$
\left(S_0(\uup,0)-\Id\right)(\d_\xi q_0(0)-\iD\uk V^{\d_k\uu}(0))
\in\ \textrm{Span}(\{q_1(0),\cdots,q_{d_1}(0)\})\,,
$$
proving the cancellation. To make the remaining computations easier we pick $(z_1,\cdots,z_{d_1})$ such that 
$$
\d_\xi q_0(0)-\iD\uk V^{\d_k\uu}(0)\ =\ \sum_{1\leq \alpha\leq d_1} z_\alpha q_\alpha(0)
$$
and replace $(q_0(\xi),\tq_1(\xi),\cdots,\tq_{d_1}(\xi))$ with $(q_0(\xi)-\xi\sum_{\alpha=1}^{d_1}z_\alpha q_\alpha(\xi),\tq_1(\xi)+\xi z_1 \tq_0(\xi),\cdots,\tq_{d_1}(\xi)+\xi z_{d_1} \tq_0(\xi))$. This ensures $\d_\xi q_0(0)=\iD\uk V^{\d_k\uu}(0)$.

With this in hands we may also compute, for $1\leq\alpha,\beta\leq d_1$,
$$
\begin{array}{rcl}
(\Omega^{(1)})_{\alpha,\beta}&=&\displaystyle
\tfrac{1}{\iD\uk}\langle \d_\xi\tq_\alpha(0),\left(S_0(\uup,0)-\Id\right)q_\beta(0)\rangle
\,+\,\tfrac{1}{\uk}\langle \tq_\alpha(0),S^{(1)}(\uup,0) q_\beta(0)\rangle\\[1em]
&=&\displaystyle
\tfrac{\uup\d_{\bM_\beta}\omega}{\iD\uk}\langle \d_\xi\tq_\alpha(0),q_0(0)\rangle
\,+\,\tfrac{1}{\uk}\langle \tq_\alpha(0),S^{(1)}(\uup,0) q_\beta(0)\rangle\\[1em]
&=&\displaystyle
-\tfrac{\uup\d_{\bM_\beta}\omega}{\iD\uk}\langle\tq_\alpha(0),\d_\xi q_0(0)\rangle
\,+\,\tfrac{1}{\uk}\langle \tq_\alpha(0),S^{(1)}(\uup,0) q_\beta(0)\rangle\\[1em]
&=&\displaystyle
-\uup\d_{\bM_\beta}\omega\langle\tq_\alpha(0),V^{\d_k\uu}(0)\rangle
\,+\,\tfrac{1}{\uk}\langle \tq_\alpha(0),S^{(1)}(\uup,0) q_\beta(0)\rangle\,.
\end{array}
$$
Now using that, for any $(s,t)$, 
$$
S_0(t,s)V^{\d_{\bM_\beta}\uu}(s)\ =\ V^{\d_{\bM_\beta}\uu}(t)+(t-s)\d_{\bM_\beta}\omega V^{\d_\zeta\uu}(t)\,,
$$
we go further to obtain
$$
\tfrac{1}{\uk}\langle \tq_\alpha(0),S^{(1)}(\uup,0) q_\beta(0)\rangle
\ =\ \int_0^\uup\langle V^{e_\alpha}(t),A^{(1)}(t) 
[V^{\d_{\bM_\beta}\uu}(t)+t\d_{\bM_\beta}\omega V^{\d_\zeta\uu}(t)]\rangle\ \dD t\
$$
with
$$
\begin{array}{rcl}
\displaystyle
\int_0^\uup\langle V^{e_\alpha}(t),A^{(1)}(t) V^{\d_{\bM_\beta}\uu}(t)\rangle\ \dD t
&=&\displaystyle
\uup\,\langle  e_\alpha,(\cL^{(1)}-\uc)\d_{\bM_\beta}\uu\rangle\\[1em]
&=&\displaystyle
\uup\,\int_0^1 \eta e_\alpha\cdot\dD f(\uu(\zeta))\d_{\bM_\beta}\uu(\zeta)\dD\zeta\\[1em] 
&=&\displaystyle
-\uup\,\d_{\bM_\beta} F_\alpha(\uk,\uM)
\end{array}
$$
and, by using $(\cL^{(1)}-\uc)\d_\zeta \uu=\d_k\omega\d_\zeta\uu-\cL\d_k\uu$ and integrating by part,
$$
\begin{array}{rcl}
\displaystyle
\int_0^\uup t\,\langle V^{e_\alpha}(t),A^{(1)}(t) V^{\d_\zeta\uu}(t)]\rangle\ \dD t
&=&\displaystyle
-\uk\int_0^\uup t\,\langle V^{e_\alpha}(t),V^{\cL\d_k\uu}(t)]\rangle\ \dD t\\[1em]
&=&\displaystyle
\uk\int_0^\uup t\,\langle V^{e_\alpha}(t),\left(\tfrac{\dD}{\dD t}-A_0(t))\right)V^{\d_k\uu}(t)]\rangle\ \dD t\\[1em]
&=&\displaystyle
\uk\,\uup\, \langle V^{e_\alpha}(0),V^{\d_k\uu}(0)\rangle\,.
\end{array}
$$
Hence, for $1\leq\alpha,\beta\leq d_1$,
$$
(\Omega^{(1)})_{\alpha,\beta}\ =\ -\uup\,\d_{\bM_\beta} F_\alpha(\uk,\uM)\,.
$$

Lastly we compute, for $1\leq\alpha\leq d_1$,
$$
\begin{array}{rcl}
(\Omega^{(2)})_{\alpha,0}
&=&\displaystyle
\tfrac{1}{(\iD\uk)^2}\langle \d_\xi\tq_\alpha(0),\left(S_0(\uup,0)-\Id\right)\d_\xi q_0(0)\rangle
\,+\,\tfrac{1}{\iD\uk^2}\langle \tq_\alpha(0),S^{(1)}(\uup,0) \d_\xi q_0(0)\rangle\\[1em]
&&\displaystyle
+\,\tfrac{1}{\iD\uk^2}\langle \d_\xi\tq_\alpha(0),S^{(1)}(\uup,0)\,q_0(0)\rangle
\,+\,\tfrac{1}{\uk^2}\langle \tq_\alpha(0),S^{(2)}(\uup,0)\,q_0(0)\rangle\,.
\end{array}
$$
On one hand, integrating by parts,
$$
\begin{array}{rcl}
\tfrac{1}{\iD\uk^2}\hspace{-1em}&\hspace{-1em}&\hspace{-1em}\langle \d_\xi\tq_\alpha(0),S^{(1)}(\uup,0)\,q_0(0)\rangle
\ =\ \displaystyle
\tfrac{1}{\iD\uk^2}\int_0^\uup \langle \d_\xi\tq_\alpha(0),S_0(\uup,t)A^{(1)}(t)V^{\d_\zeta\uu}(t)\rangle\,\dD t\\[1em]
&=&\displaystyle
\tfrac{1}{\iD\uk}\int_0^\uup \langle \d_\xi\tq_\alpha(0),S_0(\uup,t)\d_k\omega\,V^{\d_\zeta\uu}(t)\rangle\,\dD t
\,+\,\tfrac{1}{\iD\uk}
\int_0^\uup \langle \d_\xi\tq_\alpha(0),S_0(\uup,t)\left(\tfrac{\dD}{\dD t}-A_0(t)\right)V^{\d_k\uu}(t)\rangle\,\dD t
\\[1em]
&=&\displaystyle
\tfrac{1}{\iD\uk}\left[\uup\d_k\omega\,\langle \d_\xi\tq_\alpha(0),V^{\d_\zeta\uu}(0)\rangle
-\tfrac{1}{\iD\uk}\langle\d_\xi\tq_\alpha(0),\left(S_0(\uup,0)-\Id\right)\d_\xi q_0(0)\rangle\right]\,.
\end{array}
$$
On the other hand, using \eqref{e:k-S},
$$
\begin{array}{rcl}
\tfrac{1}{\uk^2}\hspace{-1em}&\hspace{-1em}&\hspace{-1em}\langle \tq_\alpha(0),[\tfrac{1}{\iD}S^{(1)}(\uup,0)\d_\xi q_0(0)
+S^{(2)}(\uup,0)\,q_0(0)]\rangle\\[1em]
&=&\displaystyle
\int_0^\uup \langle  V^{e_\alpha}(t),A^{(1)}(t)\left[V^{\d_k\uu}(t)+t\d_k\omega\,V^{\d_\zeta\uu}(t)\right]\rangle\,\dD t
\,+\,\tfrac{1}{\uk}\int_0^\uup \langle  V^{e_\alpha}(t),A^{(2)}(t)\,V^{\d_\zeta\uu}(t)\rangle\,\dD t\\[1em]
&=&\displaystyle
\uup\,\langle  e_\alpha,(\cL^{(1)}-\uc)\d_k\uu+\cL^{(2)}\d_\zeta\uu\rangle
\,+\,\uup\,\d_k\omega\ \tfrac{1}{\iD\uk}\langle \tq_\alpha(0),\d_\xi q_0(0)\rangle\,.
\end{array}
$$
Hence
$$
(\Omega^{(2)})_{\alpha,0}\ =\ 
\uup\,\langle  e_\alpha,(\cL^{(1)}-\uc)\d_k\uu+\cL^{(2)}\d_\zeta\uu\rangle
\ =\ -\uup\,\d_k F_\alpha(\uk,\uM)
$$
so that as expected
$$
\tilde\Omega_0\ =\ \uup\,\left(\begin{array}{c|ccc}\d_k\omega
&\d_{\bM_1}\omega&\cdots&\d_{\bM_{d_1}}\omega\\[0.5em]\hline &&&\\[-0.5em]
-\d_k F_\alpha(\uk,\uM)&-\d_{\bM_1} F_\alpha(\uk,\uM)&\cdots&-\d_{\bM_{d_1}} F_\alpha(\uk,\uM)\end{array}\right)\,.
$$
\end{proof}

\subsection{Hamiltonian case}

Now we turn our attention to the role of \eqref{WhithamHAM} in the spectral analysis of the linearization of \eqref{disHam}. Our analysis hereafter follows closely the proof of the foregoing subsection. The only difference stems from the presence of one extra "hidden" conservation law.

\begin{theorem}\label{propHAM}
Assuming that condition \ref{A_hamilton} holds in a neighborhood of $\uu^{\uk,\ubM,\uE}$, there exist positive $\epsilon_0$ and $\xi_0\in(0,\pi/N)$ and continuous curves $\lambda_\alpha:[-\xi_0,\xi_0]\to\C$,  $\alpha =1,\cdots, d+2$, such that
$$
\forall \xi\in[-\xi_0,\xi_0]\,,\ \sigma(S_\xi(1/\omega^{\uk,\ubM},0))\,\cap\,B(1,\epsilon_0)\ =\ \{\lambda_1(\xi),\cdots,\lambda_{d+2}(\xi)\}\,.
$$
Moreover, for $\xi\in[-\xi_0,\xi_0]$ and $\alpha =1,\cdots, d+2$
\be\label{eig_expand_HAM}
\lambda_\alpha(\xi)\ \stackrel{\xi\to0}{=}\ \eD^{\frac{1}{\omega^{\uk,\ubM}}\,\iD\uk\xi\,a_\alpha}\ +\ \cO(|\xi|^2)\,,
\ee
where $a_1,\cdots,a_{d+2}$ are the characteristic speeds of the Whitham system \eqref{WhithamHAM} linearized about parameters $(\uk,\ubM,\uE)$.
\end{theorem}


\begin{proof}
As in the previous subsection, perturbatively we may obtain dual bases $(q_0(\xi),\cdots,q_{d+1}(\xi))$ and $(\tq_0(\xi),\cdots,\tq_{d+1}(\xi))$ of the generalized eigenspaces of $S_\xi(1/\omega^{\uk,\ubM,\uE},0)$ and $S_\xi(1/\omega^{\uk,\ubM,\uE},0)^*$ associated with their spectra in $B(1,\epsilon_0)$ as long as $|\xi|\leq \xi_0$ provided $\xi_0$ and $\epsilon_0$ are small enough. Here however we start from 
$$(q_0(0),\cdots,q_{d+1}(0))\ =\ (V^{\d_\zeta\uu},V^{\d_{\bM_1}\uu},\cdots,V^{\d_{\bM_d}\uu},V^{\d_E\uu})$$ and
$$(\tq_0(0),\cdots,\tq_{d+1}(0))\ =\ (\uk V^{\uad},\uk V^{e_1},\cdots,\uk V^{e_d},\uk V^{\delta_k\cH [\uu]})\,,$$
where $\uad$ belongs to the generalized kernel of $\cL^*$, is orthogonal to $\d_{\bM_1}\uu,\cdots,\d_{\bM_{d}}\uu,\d_E\uu$ and satisfies $\langle\uad,\d_\zeta\uu\rangle=1$. Then, for $|\xi|\leq\xi_0$, the spectrum of $S_\xi(1/\omega^{\uk,\ubM,\uE},0)-\Id$ in $B(0,\epsilon_0)$ is the spectrum of the matrix
$$
\Omega_\xi\ =\ \left[\left\langle\tq_\alpha(\xi), \left(S_\xi(1/\omega^{\uk,\ubM,\uE},0)-\Id\right)q_\beta(\xi)\right\rangle\right]_{0\leq\alpha,\beta\leq d+1}\,.
$$

The cancellation discussed in the previous subsection shows that the matrix $\tilde\Omega_\xi$ defined by
$$
\tilde\Omega_\xi\ :=\ \frac{1}{\iD\uk\xi}\Sigma_\xi^{-1}\Omega_\xi\Sigma_\xi
\qquad\textrm{with}\qquad
\Sigma_\xi\ =\ 
\begin{pmatrix}(\iD\uk\xi)^{-1}&O_{1\times d+1}\\
0_{d+1\times 1}&\I_{d+1\times d+1}\end{pmatrix}
$$
depends analytically on $\xi$. The point is again to connect $\tilde\Omega_0$ with our averaged system, here~\eqref{WhithamHAM}. Manipulations as above already yield
$$
\tilde\Omega_0\ =\ \uup\,\left(\begin{array}{c|ccc|c}
\d_k\omega&\d_{\bM_1}\omega&\cdots&\d_{\bM_{d}}\omega&\d_{E}\omega\\[0.5em]\hline &&&&\\[-0.5em]
-\bB\d_k F&-\bB\d_{\bM_1} F&\cdots&-\bB\d_{\bM_{d}} F&-\bB\d_{E} F
\\[0.5em]\hline &&&&\\[-0.5em]
\frac1\uup(\Omega^{(2)})_{d+1,0}&\frac1\uup(\Omega^{(1)})_{d+1,1}&\cdots&\frac1\uup(\Omega^{(1)})_{d+1,d}&\frac1\uup(\Omega^{(1)})_{d+1,d+1}
\end{array}\right)\,,
$$
where $\uup=1/\omega^{\uk,\ubM,\uE}$, for $\alpha=1,\cdots,d+1$,
$$
\begin{array}{rcl}
(\Omega^{(2)})_{d+1,0}
&=&\displaystyle
\tfrac{1}{(\iD\uk)^2}\langle \d_\xi\tq_{d+1}(0),\left(S_0(\uup,0)-\Id\right)\d_\xi q_0(0)\rangle
\,+\,\tfrac{1}{\iD\uk^2}\langle \tq_{d+1}(0),S^{(1)}(\uup,0) \d_\xi q_0(0)\rangle\\[1em]
&&\displaystyle
+\,\tfrac{1}{\iD\uk^2}\langle \d_\xi\tq_{d+1}(0),S^{(1)}(\uup,0)\,q_0(0)\rangle
\,+\,\tfrac{1}{\uk^2}\langle \tq_{d+1}(0),S^{(2)}(\uup,0)\,q_0(0)\rangle\,,
\end{array}
$$
and, for $\beta=1,\cdots,d+1$,
$$
(\Omega^{(1)})_{d+1,\beta}\ =\ 
\tfrac{1}{\iD\uk}\langle \d_\xi\tq_{d+1}(0),\left(S_0(\uup,0)-\Id\right)q_\beta(0)\rangle
\,+\,\tfrac{1}{\uk}\langle \tq_{d+1}(0),S^{(1)}(\uup,0) q_\beta(0)\rangle\,.
$$
Moreover we may enforce $\d_\xi q_0(0)=\iD\uk\d_k\uu$.

For $\beta=1,\cdots,d+1$, setting 
$$
a_\beta\ =\ \begin{cases}\bM_\beta&\textrm{ if }1\leq\beta\leq d\\E&\textrm{ if }\beta=d+1\end{cases}
$$
and proceeding as in the previous subsection we derive
$$
\begin{array}{rcl}
(\Omega^{(1)})_{d+1,\beta}
&=&\displaystyle
\tfrac{\uup}{\iD\uk}\,\d_{a_\beta}\omega\ \langle \d_\xi\tq_{d+1}(0),q_0(0)\rangle
\ +\ \uup\ \langle \delta_k\cH [\uu],(\cL^{(1)}-\uc)\d_{a_\beta}\uu\rangle\\[1em]
&&\displaystyle
+\ \uk\,\d_{a_\beta}\omega\int_0^\uup t\,\langle V^{\delta_k\cH [\uu]}(t),\left(\tfrac{\dD}{\dD t}-A_0(t)\right)V^{\d_k\uu}(t)]\rangle\ \dD t\,.
\end{array}
$$
Now using \eqref{e:d-Ham} and \eqref{e:d-Hamflux} we also obtain
$$
\uup\ \langle \delta_k\cH [\uu],(\cL^{(1)}-\uc)\d_{a_\beta}\uu\rangle
\ =\ 
\uup\,\d_{a_\beta}S\ +\ \uk\,\d_{a_\beta}\omega\int_0^\uup \langle V^{\delta_k\cH [\uu]}(t),V^{\d_k\uu}(t)]\rangle\ \dD t\,.
$$
From this stems as expected
$$
(\Omega^{(1)})_{d+1,\beta}\ =\ 
\tfrac{\uup}{\iD\uk}\,\d_{a_\beta}\omega\ \left[\langle \d_\xi\tq_{d+1}(0),q_0(0)\rangle
\,+\,\langle \tq_{d+1}(0),\d_\xi q_0(0)\rangle\right]
\,+\,\uup\,\d_{a_\beta}S\ =\ \uup\,\d_{a_\beta}S\,.
$$

Lastly, on one hand, we still have
$$
\begin{array}{rcl}
\tfrac{1}{\iD\uk^2}\langle \d_\xi\tq_{d+1}(0)\,,\hspace{-1em}&\hspace{-1em}&\hspace{-1em}
S^{(1)}(\uup,0)\,q_0(0)\rangle\\[1em]
&=&\displaystyle
\tfrac{1}{\iD\uk}\left[\uup\,\d_k\omega\,\langle \d_\xi\tq_{d+1}(0),q_0(0)\rangle
-\tfrac{1}{\iD\uk}\langle\d_\xi\tq_{d+1}(0),\left(S_0(\uup,0)-\Id\right)\d_\xi q_0(0)\rangle\right]\,.
\end{array}
$$
On the other hand,
$$
\begin{array}{rcl}
\tfrac{1}{\uk^2}\hspace{-1em}&\hspace{-1em}&\hspace{-1em}\langle \tq_\alpha(0),[\tfrac{1}{\iD}S^{(1)}(\uup,0)\d_\xi q_0(0)
+S^{(2)}(\uup,0)\,q_0(0)]\rangle\\[1em]
&=&\displaystyle
\int_0^\uup \langle  V^{e_\alpha}(t),A^{(1)}(t)\left[V^{\d_k\uu}(t)+t\d_k\omega\,V^{\d_\zeta\uu}(t)\right]\rangle\,\dD t
\,+\,\tfrac{1}{\uk}\int_0^\uup \langle  V^{e_\alpha}(t),A^{(2)}(t)\,V^{\d_\zeta\uu}(t)\rangle\,\dD t\\[1em]
&=&\displaystyle
\uup\,\langle  e_\alpha,(\cL^{(1)}-\uc)\d_k\uu+\cL^{(2)}\d_\zeta\uu\rangle
\,+\,\uk\,\d_k\omega\,\int_0^\uup t\,\langle V^{\delta_k\cH [\uu]}(t),\left(\tfrac{\dD}{\dD t}-A_0(t)\right)V^{\d_k\uu}(t)]\rangle\ \dD t\\[1em]
&=&\displaystyle
\uup\,\d_kS
\,+\,\uup\,\d_k\omega\ \tfrac{1}{\iD\uk}\langle \tq_\alpha(0),\d_\xi q_0(0)\rangle\,.
\end{array}
$$
Hence as expected $(\Omega^{(2)})_{d+1,0}\ =\ \uup\,\d_kS$.

This achieves the proof of
$$
\tilde\Omega_0\ =\ \uup\,\left(\begin{array}{c|ccc|c}
\d_k\omega&\d_{\bM_1}\omega&\cdots&\d_{\bM_{d}}\omega&\d_{E}\omega\\[0.5em]\hline &&&&\\[-0.5em]
-\bB\d_k F&-\bB\d_{\bM_1} F&\cdots&-\bB\d_{\bM_{d}} F&-\bB\d_{E} F
\\[0.5em]\hline &&&&\\[-0.5em]
\d_k S&\d_{\bM_1} S&\cdots&\d_{\bM_d}S&\d_{E}S
\end{array}\right)
$$
that yields the theorem.
\end{proof}

\section{Conclusions and Remarks}

We have proved that the slow modulation \emph{ansatz} captures accurately the essential spectral features of the space-time low\footnote{In the sense of \emph{close to $1$}.} Floquet-multiplier evolution linearized about a periodic wave whose period belongs to the lattice. As already mentioned this opens at least three classes of questions.

In many cases one also expects that in the large-time the dynamics is effectively reduced to a slow modulation evolution. Our analysis provides the spectral background --- including, as appears from our proofs, expansion of critical eigenfunctions --- for a proof of such claims, either in the linear regime as, for continuous systems, in \cite{R_linKdV} for the Korteweg--de Vries equation, or in the complete nonlinear regime, as in \cite{JNRZ-conservation} for general parabolic systems of partial differential equations.

With this in mind, one expects to be able to identify on averaged systems the key features of the discrete dynamics. It is then natural to ask if one may design discrete systems leading to desired modulation properties, either with applications to the analysis of numerical schemes where the goal is to preserve the slow modulation asymptotics at the discrete level, or in the effective conception of smart materials, a fast-growing field of investigation.

At last, even for the direct problem at the spectral level, remains the technical challenge of dealing with waves of \emph{general} period, a seemingly quasi-periodic problem.

\medskip
{\bf Acknowledgement}: The main part of this work was carried out during the $3$-months stay of B.K. in Lyon on 2013. B.K. would like to thank Sylvie Benzoni-Gavage for the kind invitation that was at the origin of this stay, and acknowledge the hospitality of Institut Camille Jordan at the Universit\'e Lyon 1 and the financial support of the CMIRA ACCUEIL DOC program of the region Rh\^one-Alpes.


\bibliographystyle{alpha}
\bibliography{Ref} 

\end{document}